\newcommand*\mR{\mathbb{R}}

\newcommand*\mN{\mathbb{N}}

\newcommand*\esom{L^{\infty}(\Omega)}

\newcommand*\porn{\int\limits_{\mR^n}}
\newcommand*\hkrn{H_\nu^{\Omega}(\mR^n)}

\newcommand*\poom{\int\limits_{\Omega}}
\newcommand*\poza{\int\limits_{\Omega^c}}
\newcommand*\homc{H_\nu(\Omega^c)}

\newcommand*\rnbeze{\lim\limits_{\varepsilon \to 0^+} \int\limits_{\mR^n \backslash B(0,\varepsilon)}}
\newcommand*\bezomkw{\iint\limits_{\mR^{2n}\backslash \Omega^C \times \Omega^C}}

\newcommand*\vkrn{V_\nu^{\Omega}(\mR^n)}
\newcommand*\hrn{H_\nu(\mR^n)}
\newcommand*\wtu{\widetilde{u}}
\newcommand*\mP{\mathbb{P}}
\newcommand*\mE{\mathbb{E}}
\newcommand*\wtx{\widetilde{x}}
\newcommand*\wty{\widetilde{y}}
\newcommand*\wtg{\widetilde{g}}
\newcommand*\eps{\varepsilon}
\newcommand*\pow{\int\limits_{W}}
\newcommand*\potw{\int\limits_{TW}}
\newcommand*\hnd{H_\nu^D(\mR^n)}
\newcommand*\vnd{V_\nu^D(\mR^n)}
\newcommand*\pd{\int\limits_D}
\newcommand*\pdc{\int\limits_{D^c}}
\documentclass[reqno,10pt]{amsart}
\usepackage{amsthm}
\usepackage[utf8]{inputenc}
\usepackage[T1]{fontenc}
\theoremstyle{plain}
\newtheorem{thm}{Theorem}[section]
\newtheorem{prop}[thm]{Proposition}
\newtheorem{lem}[thm]{Lemma}	
\newtheorem{cor}[thm]{Corollary}
\newtheorem{defi}[thm]{Definition}
\theoremstyle{definition}
\newtheorem{exa}[thm]{Example}
\newtheorem*{rem}{Remark}
\usepackage[top=2.5cm, bottom=2.5cm, left=2.5cm, right=2.5cm]{geometry}
\usepackage{lipsum}
\usepackage{amsmath}
\usepackage{amsfonts}
\usepackage{amssymb}
\usepackage{graphicx}
\usepackage[usenames]{color}
\usepackage{wasysym}
\usepackage{enumerate}
\usepackage{hyperref}
\usepackage{tikz}
\usepackage{float}
\usepackage{pdfpages}
\usepackage{tocvsec2}

\DeclareMathOperator*{\dist}{dist}

\numberwithin{equation}{section}
\begin{document}
\title{The Dirichlet problem for nonlocal L\'evy-type operators}
\author{Artur Rutkowski}
\address{Artur Rutkowski, Faculty of Pure and Applied Mathematics, Wrocław University of Science and Technology, Wybrzeże Wyspiańskiego 27, 50-370 Wrocław, Poland}
\email{artur.rutkowski@pwr.edu.pl}
\thanks{The author is supported by the National Science Center (Poland) grant: DEC-2014/14/M/ST1/00600.}
%\begin{center}\textbf{\large{BOUNDARY PROBLEMS FOR NONLOCAL OPERATORS}}
	%Artur Rutkowski\footnote{\noindent The author is supported by the National Science Center (Poland) grant: DEC-2014/14/M/ST1/00600. \\Key words and phrases: Dirichlet problem, L\'evy measure, extension operator, maximum principle, weak solutions \\E-mail address: \textbf{artur.rutkowski@pwr.edu.pl}}, Faculty of Pure and Applied Mathematics, Wrocław University of Science and Technology, Poland.\\
%\end{center}
\begin{abstract}
	We present the theory of the Dirichlet problem for nonlocal operators which are the generators of general pure-jump symmetric L\'evy processes whose L\'evy measures need not be absolutely continuous. We establish basic facts about the Sobolev spaces for such operators, in particular we prove the existence and uniqueness of weak solutions. We present strong and weak variants of maximum principle, and $L^\infty$ bounds for solutions. We also discuss the related extension problem in $C^{1,1}$ domains.
\end{abstract}

\maketitle

\section*{Introduction}
 We present results on existence, uniqueness and regularity of solutions to both weak and strong versions of the Dirichlet problem for nonlocal L\'{e}vy-type operators. Let $\nu$ be a nonnegative Borel measure on $\mR^n$, satisfying
\begin{equation}\label{Levy}
\nu(\{0\}) = 0, \ \ \ \nu(-A) = \nu(A), \ \ \ \hbox{and } \porn (1 \wedge |y|^2) d\nu(y) < \infty,
\end{equation}
for every Borel set $A\subseteq \mR^n$. $\nu$ is called the L\'evy measure. For the operator of the form
\begin{equation}\label{PV}
Lu(x) = \hbox{PV} \porn (u(x) - u(x+y))d\nu(y) := \rnbeze (u(x) - u(x+y))d\nu(y),
\end{equation}
we consider the weak version of the following ``boundary'' value problem:
\begin{equation}\label{DP}
\begin{cases}
Lu = f & \hbox{in } \Omega, \\
u = g & \hbox{in } \mR^n \backslash \Omega,
\end{cases}
\end{equation}
where $\Omega\subseteq \mR^n$ is nonempty, open, and bounded, and $f,g$ are given real functions.\\
\indent
There are important reasons to study operators of the form \eqref{PV}. One of them is the Courr\'ege theorem, which characterizes the operators satisfying the maximum principle, see \cite{Courr}, \cite{Schilling} - \eqref{PV} forms a representative subclass of such operators. Another reason is in the modeling of the real world phenomena, see \cite{RosOton} and the references therein.

The purpose of this article is to analyse the Dirichlet problem in detail, for the operators of the form $\eqref{PV}$, without assuming the absolute continuity of the measure $\nu$. The case of L\'evy measures with densities (and possible $x$-dependece) was investigated in the article by Felsinger, Kassmann and Voigt \cite{FKV}. The choice of topics which are included in our work is partly inspired by the survey paper of Ros-Oton \cite{RosOton}. We cover the technical details omitted in that paper, and extend the results to arbitrary symmetric L\'evy measures. We also discuss the related extension problem for Sobolev spaces.

We proceed as follows. In Section \ref{sec:operator} we give the basic facts about the operator $L$. Section \ref{sec:prob} shows how the problem \eqref{DP} can be used to study L\'evy processes. In Section \ref{sec:func} we introduce the quadratic form of the operator $L$. Domains of such forms - generalized Sobolev spaces will serve as the framework for the notion of weak solutions. The definition of these spaces - $\vkrn$ follows \cite{FKV}. In Section \ref{sec:weak} we present the results on weak solutions: existence, uniqueness, stability and connection with strong solutions. All of these facts are proved for general symmetric L\'evy measures. In the process of proving the existence and uniqueness result (Theorem \ref{exigen}), we establish the Poincar\'e inequality for arbitrary symmetric L\'evy measures (Theorem \ref{poingen}). The reader may find its proof interesting. We first show that the quadratic forms can be represented by the means of forms with discrete L\'evy measures (Lemma \ref{delty}). Then we show that the inequality holds for the atomic L\'evy measures - they let us effectively grasp the notion of jumping out of the set (Lemma \ref{poi}). In Section \ref{sec:comp} we prove the strong and weak versions of the maximum principle and obtain $L^\infty$ bounds for solutions using barriers. In Section \ref{sec:ext} we use an elementary geometric method to define the extension operator for isotropic, absolutely contiunous L\'evy measures with a mild scaling condition in $C^{1,1}$ domains. This operator turns out to be continuous between appropriate function spaces, as we argue in Theorem \ref{ext}. So far, this topic has been studied for the classical fractional Sobolev spaces, see \cite{Jonsson1978} and \cite{Zhou}, and fractional Sobolev spaces with relaxed exterior conditions \cite{BDMK}. See also, the article by Valdinoci et al. \cite{Sobolev}. The extension problem is strictly linked to the solvability of the Dirichlet problem, in accordance with the exterior condition $g$, cf. Corollary \ref{exicons}.\\
\indent
Irregular L\'evy measures caught some interest lately in the context of PDEs, see e.g. \cite{Jakobsen}. In the area of stochastic processes, operators with singular L\'evy measures can be used to investigate the processes whose jump intensity fails to have a density, e.g. processes with independent coordinates.
\\

\section{Preliminaries}\label{sec:operator}
We will use the following notation:
\linespread{0.25}
\begin{itemize}
	\item $Y^*$ - dual space of a Banach space $Y$,
	\item $(f,g) := \int\limits_{\Omega} f(x) g(x) dx$ - the scalar product in $L^2(\Omega)$,
	\item $(f,g)_D := \int\limits_{D} f(x) g(x) dx$ - the scalar product in $L^2(D)$ for other open sets $D$,
	\item $\langle f,g\rangle_\nu$ - see \eqref{DF},
	\item $x \wedge y = \min\{x,y\}$,
	\item $x \vee y = \max\{x,y\}$,
	\item $C_0(X)$ - continuous functions on a locally compact topological space $X$, vanishing at infinity (in the sequel $X$ will be an open subset of $\mR^n$ with the Euclidean topology),
	\item $C^n(X)$ - $n$ times continuously differentiable functions,
	\item $C_b^n(X)$ - functions from $C^n(X)$, with bounded derivatives of order up to $n$,
	\item $C^{\infty}(X)$ - infinitely many times continuously differentiable functions,
	\item $C^{\infty}_c(X)$ - functions from $C^{\infty}(X)$ with compact support,
	\item $L^0(X)$ - Borel measurable functions on $X$,
	\item $L^p(X)$ - equivalence classes (w.r.t. being equal a.e.) of functions with finite $L^p$-norm,
	\item $L^{\infty}(X)$ - equivalence classes of functions with finite essential supremum norm,
	%\item $Lip(X)$ - (globally) Lipschitz functions on $X$,
	%\item $Lip_c(X)$ - compactly supported Lipschitz functions on $X$,
	\item $V_\nu^{D}(\mR^n)$ - see Definition \ref{hkomega},
	\item $H_\nu^{D}(\mR^n)$ - see Definition \ref{hkomega},
	\item $H_\nu(\mR^n)$ - see Definition \ref{hkomega},
	\item $D(L,x)$ - the set of functions $u$, for which $Lu(x)$ exists, $D(L,A) = \bigcap_{x\in A}D(L,x)$,
	\item $\nu_x$ - the shift of measure $\nu$ by $x$: $\nu_x(A) = \nu(A-x)$.
\end{itemize}
\linespread{1.25}

For every function space above we only consider real functions. When we write a.e. (almost everywhere) we mean the Lebesgue measure, unless stated otherwise. We would like to emphasize, that $\Omega$ is always a fixed nonempty, bounded, open set. Arbitrary open sets are usually denoted by the letter $D$.\\ \\
The results of this section are mostly well-known, however we present them for the sake of completeness of the presentation.\\
By \eqref{Levy}, $\nu$ is $\sigma$-finite. The symmetry yields $\porn u(x) d\nu(x) = \porn u(-x) d\nu(x)$, which we will often use without mention.
The nonlocality of the operator \eqref{PV} means that in order to compute $Lu$ at $x\in\Omega$, we use the values of $u$ in the support of $\nu_x$, i.e. possibly far from $x$, while the local operators (e.g., $\nabla$ or $\Delta$) only require the values from an arbitrarily small neighborhood of $x$. In what follows, we stipulate that $f\in L^2(\Omega)$, and refrain from making further assumptions on $f$ and $g$ until we reformulate the problem \eqref{DP} in the framework of Hilbert spaces in Sections \ref{sec:func} and \ref{sec:weak}. \\
By changing variables in \eqref{PV}, we obtain the following alternative form of the operator:
$$Lu(x) = \hbox{PV} \porn (u(x) - u(y))d\nu_x(y).$$
Note that for every $u\in C_b(\mR^n)$ and $\varepsilon > 0$, we have 
\begin{equation}\label{finit}
\int\limits_{B(0,\varepsilon)^c} |u(x) - u(x+y)|d\nu(y) \leq 2 \nu(B(0,\varepsilon)^c)  \|u\|_{\infty} < \infty.
\end{equation}
Our formula for the operator $L$ is pointwise, and it may depend on the value of the function in a single point. This is because the measure $\nu$ is not necessarily absolutely continuous. Therefore, the formula \eqref{PV} may yield different results for functions that are equal almost everywhere. This problem can be managed by considering $L$ in the global sense, as an operator on a function space.
\begin{prop}\label{glob}
If the functions $u,v$ are measurable, $u=v$ a.e. in $\mR^n$, and $Lu,Lv$ are well defined a.e. in $\Omega$, then $\|Lu - Lv\|_{L^2(\Omega)} = 0$, hence $Lu = Lv$ a.e. in $\Omega$. 
\end{prop}
\begin{proof}
Since $Lu,Lv$ are well-defined and finite a.e., we have
\begin{equation}\label{ll2}
\poom (Lu(x) - Lv(x))^2 dx = \poom  \lim\limits_{\eps \to 0^+} \left(\int\limits_{B(0,\eps)^c} (u-v)(x) - (u-v)(x+y) d\nu(y)\right)^2 dx.
\end{equation}
Using H\"{o}lder's inequality, the monotone convergence therorem, and Fubini's theorem, we can estimate \eqref{ll2} as follows 
\begin{align*}
&\poom  \lim\limits_{\eps \to 0^+} \left(\int\limits_{B(0,\eps)^c} (u-v)(x) - (u-v)(x+y) d\nu(y)\right)^2 dx\\
&\leq \poom\lim\limits_{\eps \to 0^+} \nu(B(0,\eps)^c) \int\limits_{B(0,\eps)^c} ((u-v)(x) - (u-v)(x+y))^2d\nu(y) dx\\
&= \lim\limits_{\eps \to 0^+} \poom \nu(B(0,\eps)^c)\int\limits_{B(0,\eps)^c} ((u-v)(x) - (u-v)(x+y))^2 d\nu(y) dx\\
& = \lim\limits_{\eps\to 0^+} \int\limits_{B(0,\eps)^c}\nu(B(0,\eps)^c) \poom ((u-v)(x) - (u-v) (x+y))^2 dx d\nu(y).
\end{align*}
Since the inner integral is equal to $0$ for every $y\in \mR^n$, the proposition is proved.
\end{proof}
The next result gives an insight into the domain of $L$.
\begin{prop}\label{smthint}
If $u\in C^2_b(\mR^n)$, then $Lu(x)$ is well defined for $x\in\mR^n$, and $Lu \in L^2(\Omega)$.
\end{prop}
\begin{proof}
Let $u\in C^2_b(\mR^n)$. Substituting $-y$ for $y$ in \eqref{PV} and adding side by side gives
\begin{equation}\label{C2}
Lu(x) = \frac 12 \rnbeze (2u(x) - u(x+y) - u(x-y)) d\nu(y).
\end{equation}
By Taylor's expansion, for $x,y \in \mR^n$: 
\begin{align*}2u(x) - u(x+y) - u(x-y) = 2u(x) &- \left[u(x) + y\circ \nabla u(x) + \sum\limits_{i,j=1}^n\frac{\partial^2 u (\xi)}{\partial x_i \partial x_j} y_i y_j\right] \\&- \left[u(x) - y\circ \nabla u(x) + \sum\limits_{i,j=1}^n\frac{\partial^2 u(\xi)}{\partial x_i \partial x_j} y_i y_j\right]
\end{align*}
\begin{equation*}
\quad = -2\sum\limits_{i,j=1}^n\frac{\partial^2 u(\xi)}{\partial x_i \partial x_j} y_i y_j,
\end{equation*}
where $\xi \in B(x,|y|)$. Since $u\in C^2_b(\mR^n)$, we obtain
\begin{equation}\label{linf}|2u(x) - u(x+y) - u(x-y)| \leq C (1 \wedge |y|^2),
\end{equation}
for a number $C > 0$ independent of $x$, i.e. a constant. As a consequence, $\porn(2u(x) - u(x+y) - u(x-y)) d\nu(y)$ converges absolutely. By the dominated convergence theorem,
\begin{equation}\label{NOPV}
Lu (x) = \rnbeze (u(x) - u(x+y))d\nu(y) = \frac 12 \porn(2u(x) - u(x+y) - u(x-y)) d\nu(y).
\end{equation}
Furthermore,
\begin{align*}
\poom Lu(x)^2 dx &= \poom\left( \frac 12 \porn (2u(x) - u(x+y) - u(x-y))d\nu(y)\right)^2 dx\\
&\leq \left(\porn (1\wedge |y|^2) d\nu(y) \right)^2\poom \frac{C^2}4 dx.
\end{align*}
Since $\Omega$ is bounded, $Lu \in L^2(\Omega)$.
\end{proof}
\begin{rem}
Note that the estimate \eqref{linf} also shows that $Lu \in L^{\infty}(\Omega)$.
\end{rem}
%The following proposition shows the asymptotic behavior of the L\'{e}vy measure near 0.
%\begin{prop}
%$\nu(B(0,r)^c) = o(r^{-2})$ for $r\to 0^+$.
%\end{prop}
%\begin{proof} Let $r < 1$. We have
%\begin{align*}
%&\nu(B(0,r)^c) = \int\limits_{B(0,r)^c} d\nu(y) = \int\limits_{B(0,r)^c} \left(1\wedge \frac {|y|^2}{r^2}\right) d\nu(y) = \frac 1{r^2} \int\limits_{B(0,r)^c} (r^2\wedge |y|^2) d\nu(y)\\
%&\leq \frac 1 {r^2}\porn (r^2 \wedge |y|^2)d\nu(y).
%\end{align*}
%The latter integral converges to 0 as $r\to 0^+$, since $r^2\wedge|y|^2 \leq 1\wedge |y|^2$ which is integrable with respect to $\nu$.
%\end{proof}

\section{Connection with L\'{e}vy processes}\label{sec:prob}
In this section we will provide a probabilistic motivation for studying the Dirichlet problem \eqref{DP}, and an explanation for the assumptions in \eqref{Levy}. For further details we refer to chapters 1, 6, 8 of \cite{Sato}, and chapters I-V of \cite{Dynkin}.
\begin{defi}
We call an $\mR^n$-valued stochastic process $(X_t)_{t\geq 0}$ a L\'{e}vy process, if it is stochastically continuous and has stationary independent increments. 
\end{defi}
For an $\mR^n$-valued L\'{e}vy process we have the family of transition probabilities: $p_t(x,A) = \mathbb{P}^x(X_t \in A) = \mP(X_t\in A | X_0 = x)$. They yield a strongly continuous semigroup of contractions on $C_0(\mR^n)$: $p_t f (x) = \porn f(y) p_t(x,dy)$.
Recall that the generator $G$ of a strongly continuous semigroup of contractions $(p_t)_{t\geq 0}$ on a Banach space is
$$Gu(x) =  \lim_{t\to 0} \frac{p_t u(x) - u(x)}{t},$$
with the limit, if it exists, taken in the norm of the Banach space.
If the contraction semigroup is associated to a L\'{e}vy process, then we also say that $G$ is the generator of the process. The following result is well-known:
\begin{thm}
Let $G$ be the generator of a L\'{e}vy process. Then, for every $u \in C^2_0(\mR^n)$,
\begin{equation*}
Gu(x) = \sum\limits_{i=1}^n b_i u_{x_i}(x) + \frac 12 \sum\limits_{i,j=1}^n a_{ij} u_{x_ix_j} (x) + \porn (u(x+y) - u(x) - y\circ \nabla u(x)\cdot \textbf{\upshape{1}}_{|y|<1}) d\nu (y),
\end{equation*}
where $A = [a_{ij}]$ is a symmetric nonnegative-definite matrix, $[b_i] \in \mR^n$, and the (L\'{e}vy) measure $\nu$ satisfies $\porn (1\wedge |y|^2) d\nu (y) < \infty$.
\end{thm}
\noindent The L\'{e}vy measure can be understood as the intensity of jumps of the process $X_t$.\\
We want to discuss only pure jump processes therefore we drop the drift (first derivatives) and diffusion (second derivatives), ending up with
\begin{equation}\label{JUMP}
Gu(x) = \porn (u(x+y) - u(x) - y\circ \nabla u(x)\cdot \textbf{1}_{|y| < 1})  d\nu (y).
\end{equation}
If $\nu$ is symmetric, we obtain an operator of the form \eqref{NOPV}. Namely, we have
\begin{equation}
-Gu(x) = \frac 12 \porn (2u(x) - u(x+y) - u(x-y)) d\nu(y).
\end{equation}
Thus, $L = -G$, cf. \eqref{NOPV}. We note that $L = -G$ is positive definite.
\begin{exa} The Dirichlet problem arises when studying exit times for L\'{e}vy processes. 
Let $s(x) = \mathbb{E}^x \tau_{\Omega}$, where $\tau_{\Omega} = \inf\{t \geq 0: X_t \notin \Omega\}$ is the first exit time from a nonempty bounded open $\Omega\subset \mR^n$ for the L\'{e}vy process $(X_t)$ with the generator $G$. Then $s$ satisfies
\begin{equation}\label{exit}
\begin{cases}
-Gs = 1 & \hbox{in } \Omega, \\
s = 0 & \hbox{in } \Omega^C.
\end{cases}
\end{equation}
While the second equality is trivial, the first one requires auxiliary notions and results from probabilistic potential theory, therefore we skip the details of this connection. We will, however, make related calculations in Example \ref{exbarr}. More generally, $h(x) = \mE^x g(X_{\tau_{\Omega}}) - \mE^x \int\limits_0^{\tau_{\Omega}}f(X_t)dt$ solves \eqref{DP}. For an elegant derivation of this fact, see chapter V in Dynkin's book \cite{Dynkin}.
\end{exa}

\begin{defi}
For $\alpha \in (0,2)$ we define $C = \frac {2^\alpha\Gamma(\frac {n+\alpha}2)}{\pi^{\frac n2}|\Gamma(-\frac \alpha 2)|} $ and
$$(-\Delta)^{\frac \alpha 2}u(x) = C^{-1} \cdot \mathrm{PV} \porn \frac {u(x) - u(x+y)}{|y|^{n+\alpha}} dy.$$
\end{defi}
The generator of the isotropic $\alpha$-stable process $(\alpha \in (0,2))$ is the fractional Laplacian $-(-\Delta)^{\frac{\alpha} 2}$. Here the L\'{e}vy measure is absolutely continuous w.r.t. Lebesgue measure with the density function (kernel) $K(y) = C^{-1} \frac 1{|y|^{n+\alpha}}$. Stable processes, with their generators, are a natural nonlocal extension of the Brownian motion and its generator - the classical Laplacian. That is why these objects draw a great deal of attention of the researchers from fields of analysis, PDE-s, and stochastic processes. For more information about stable processes we refer to \cite{MR2569321}.\\
In the probabilistic context it is sometimes stressed that the L\'evy measure should span the whole $\mR^n$ space, i.e. its support should not be contained in a proper subspace of $\mR^n$. Otherwise, there would be little reason to consider the given process as a process in $\mR^n$.

\section{Function spaces}\label{sec:func}
The methods of the Hilbert spaces, in particular the quadratic forms, provide us with a convenient framework for solving the weak variant of \eqref{DP}. Before we define the appropriate function spaces, we need to conduct calculations similar to those from Proposition \ref{glob}. Let $D$ be a nonempty open set, and let the functions $u_1,u_2$ be equal a.e. in $\mR^n$. Assume that
\begin{align*}
\int\limits_D \porn (u_i(x) - u_i(x+y))^2 d\nu(y) dx < \infty, \ \ \ \ \hbox{i = 1,2}.
\end{align*}
By Fubini-Tonelli theorem we have,
\begin{align*}
&\int\limits_D \porn ((u_1 - u_2)(x) - (u_1 - u_2)(x+y))^2 d\nu(y) dx \\
&= \porn\int\limits_D ((u_1 - u_2)(x) - (u_1 - u_2)(x+y))^2 dx d\nu(y),
\end{align*}
which is equal to 0 since $u_1 = u_2$ a.e. This fact lets us operate on equivalence classes of functions, even when $\nu$ is singular. The next definition follows \cite{FKV}.
\begin{defi}\label{hkomega} For a nonempty open (not necessarily bounded) $D \subseteq \mR^n$, we define the function spaces 
\begin{align*}
V_\nu^{D}(\mR^n) &= \left\{u\in L^0(\mR^n) : \|u\|_{\vnd} < \infty\right\},
\\
H_\nu^{D}(\mR^n) &= \left\{u\in V_\nu^{D}(\mR^n): u \equiv 0 \hbox{ a.e. in }\mR^n \backslash D\right\},
\end{align*}
where
$$\|u\|_{V_\nu^{D}(\mR^n)} = \sqrt{\|u\|_{L^2(D)}^2 + \frac 12 \int\limits_D \porn (u(x) - u(y))^2 d\nu_y(x) dy}.$$
Furthermore, we let
$$H_\nu(\mR^n) := V_\nu^{\mR^n}(\mR^n).$$
In particular, $\|u\|_{\hrn} = \|u\|_{V^{\mR^n}_{\nu}(\mR^n)}$. These spaces are called the (generalized) Sobolev spaces. For $u,v \in \hrn$, we write
\begin{align}
\langle u,v\rangle_\nu &=  \frac 12 \porn\porn (u(x) - u(y))(v(x) - v(y)) d\nu_y(x) dy\label{DF} \\
&= \frac 12 \porn\porn (u(y) - u(x+y))(v(y) - v(x+y))d\nu(x) dy,\nonumber
\end{align}
so that the norm on $H_\nu(\mR^n)$ can be rewritten as $\|u\|_{H_\nu(\mR^n)} = \sqrt{\|u\|_{L^2(\mR^n)}^2 + \langle u,u\rangle_\nu}$.
\end{defi}
\noindent In literature, the expression $\|u\|_{L^2(\Omega)} + \langle u,u\rangle_\nu$ is sometimes referred to as $\mathcal{E}_1(u,u)$. $\langle u,u\rangle_\nu$ is called the quadratic form of the operator $L$. It was well studied in the context of nonlinear equations in \cite{TEJ}.\\
\noindent Note that $\hnd \subseteq L^2(\mR^n)$ and $H_\nu(\mR^n) \subseteq \vnd$ for every $D$.
The following identity will be used frequently.
\begin{lem}\label{sym}
	For every $u\in L^0(\mR^n)$,
	$$\porn \int\limits_D (u(x) - u(y))^2 d\nu_y(x) dy = \int\limits_D \porn (u(x) - u(y))^2 d\nu_y(x) dy.$$
\end{lem}
\begin{proof}
	By Fubini-Tonelli theorem, translation invariance of Lebesgue measure, and the symmetry of $\nu$, we get
	\begin{align*}
	&\porn\pd (u(x) - u(y))^2 d\nu_y(x) dy = \porn \porn (u(x+y) - u(y))^2 \textbf{1}_{D}(x+y) d\nu(x) dy \\
	&=\porn\porn (u(x+y) - u(y))^2  \textbf{1}_{D}(x+y) dy d\nu(x) = \porn \porn (u(y) - u(y-x))^2\textbf{1}_{D}(y) dy d\nu(x) \\
	&=\porn\porn (u(y) - u(y+x))^2\textbf{1}_{D}(y) dy d\nu(x) = \pd\porn (u(x+y) - u(y))^2 d\nu(x) dy\\
	&= \pd\porn (u(x) - u(y))^2 d\nu_y(x) dy.
	\end{align*}
\end{proof}
For $u\in V_\nu^D(\mR^n)$ we can easily conclude that the corresponding integrals over $D\times D^c$ and $D^c\times D$ are also equal. Reader interested in more general results of this type may consult \cite[Lemma 6.4.]{Jakobsen}.
\begin{cor}
	For every $u \in \hnd$, we have $\|u\|_{\vnd} \leq \|u\|_{H_\nu(\mR^n)} \leq 2 \|u\|_{\vnd}$, i.e. the norms are equivalent on $\hnd$. In particular, $\hnd \subseteq H_\nu(\mR^n)$  
\end{cor}
\begin{proof}
	The corresponding $L^2$ norms are identical, because $supp(u) = D$, so we will only focus on the remaining parts of the norms.\\
	The first inequality is trivial. For the second inequality, we note that for $u\in\hnd$, we have $\pdc\pdc (u(x) - u(y))^2d\nu_y(x) dy = 0$, hence by Lemma \ref{sym}
	\begin{align*}&\porn\porn (u(x) - u(y))^2d\nu_y(x) dy  = (\pd\pd + \pd\pdc + \pdc\pd + \pdc\pdc) (u(x) - u(y))^2 d\nu_y(x) dy \\
	&= (\pd\pd + \pd\pdc + \pdc\pd)(u(x) - u(y))^2d\nu_y(x) dy\\
	&= (\pd\porn + \pdc\pd) (u(x) - u(y))^2 d\nu_y(x) dy \leq 2 \pd\porn (u(x) - u(y))^2 d\nu_y(x) dy.
	\end{align*}
\end{proof}
The proof of the following result is almost identical to the analogue in \cite{FKV}. Nonetheless, we present it for the convenience of the reader.
\begin{lem}
$H_\nu(\mR^n)$ and $\hnd$ are Hilbert spaces with the inner product $(u,v)_{\mR^n} + \langle u,v \rangle_\nu.$
\end{lem}
\begin{proof}
It is enough to prove the proposition for $H_\nu(\mR^n)$, because once we establish that, it suffices to note that $\hnd = \{u\in H_\nu(\mR^n): u|_{D^c} \equiv 0\}$ is a closed subspace of $H_\nu(\mR^n)$.\\ $H_\nu(\mR^n)$ is obviously closed upon multiplication by scalars. The closedness under addition goes as follows. Let $u, v \in H_\nu(\mR^n)$. Since $(\cdot,\cdot)$ and $\langle \cdot,\cdot\rangle_\nu$ both admit the Cauchy-Schwarz inequality, we have $\|u+v\|_{H_\nu(\mR^n)} \leq \|u\|_{H_\nu(\mR^n)} + \|v\|_{H_\nu(\mR^n)}$, hence $u+v \in H_\nu(\mR^n)$.\\ We know that $(u,v)$ is an inner product, and $\langle u,v \rangle_\nu$ is bilinear, symmetric and nonnegative definite, therefore their sum is an inner product too. \\
To prove the completeness, let $(u_n)$ be a Cauchy sequence in $H_\nu(\mR^n)$. This implies that $(u_n)$ is a Cauchy sequence in $L^2(\mR^n)$, so it converges in $L^2(\mR^n)$ to some $u$. Let us choose a subsequence $(u_{n_k})$ that converges to $u$ a.e. From Fatou's lemma and the fact that $(u_n)$ is Cauchy, hence bounded in $H_\nu(\mR^n)$, we conclude that
\begin{align*}
\porn\porn (u(x)-u(y))^2 d\nu_y(x) dy &\leq \liminf\limits_{k\to\infty}\porn\porn (u_{n_k}(x) - u_{n_k}(y))^2 d\nu_y(x) dy\\ &\leq \sup\limits_{n\in\mN} \|u_n\|_{H_\nu(\mR^n)}^2 < \infty.
\end{align*}
Therefore $u\in H_\nu(\mR^n)$. Now we will prove that $u_{n_k} \mathop{\longrightarrow}\limits^{n\to\infty} u$ in $H_\nu(\mR^n)$. By Fatou's lemma:
\begin{align*}&\porn\porn (u_{n_k}(x) - u_{n_k}(y) - (u(x) - u(y)))^2 d\nu_y(x) dy\\ &\leq \liminf\limits_{l\to\infty}\porn\porn (u_{n_k}(x) - u_{n_k}(y) -(u_{n_l}(x) - u_{n_l}(y))^2 d\nu_y(x) dy.
\end{align*}
The right hand side is less than $\varepsilon$ for $k$ large enough since $(u_n)$ is Cauchy in $H_\nu(\mR^n)$. Thus, $u_{n_k}$ converges to $u$ in $H_\nu(\mR^n)$, and so $u_n \mathop{\longrightarrow}\limits^{n\to\infty} u$ in $H_\nu(\mR^n)$. That finishes the proof of completeness of $H_\nu(\mR^n)$.
\end{proof}
\begin{exa}
	If $d\nu(x) = C \frac 1{|x|^{n+\alpha}}dx$, for some $\alpha \in (0,2)$, i.e. if $L$ is the fractional Laplacian, then $\vkrn,\hkrn,\hrn$ are called the fractional Sobolev spaces. 
\end{exa}

\begin{rem}
	Our approach to the bilinear forms \eqref{DF} is straightforward in the sense that we do not use any deep results from the functional analysis. The book of Ma and R\"ockner \cite{MaRock} presents the relationship between the operator and its quadratic form in a more abstract sense, in the context of semigroups and resolvents theory.
\end{rem}
The definition of Sobolev spaces yields the following monotonicity properties.
\begin{lem}
	If $\nu_1 \leq \nu_2$, then $H_{\nu_2}^{\Omega}(\mR^n) \subseteq H_{\nu_1}^\Omega(\mR^n)$, $V_{\nu_2}^{\Omega}(\mR^n) \subseteq V_{\nu_1}^\Omega(\mR^n)$, and $H_{\nu_1}(\mR^n) \subseteq H_{\nu_2}(\mR^n)$.\\
	If $\Omega_1 \subseteq \Omega_2$, then $H_\nu^{\Omega_2}(\mR^n) \subseteq  H_\nu^{\Omega_1}(\mR^n)$, and $V_\nu^{\Omega_2}(\mR^n) \subseteq V_\nu^{\Omega_1}(\mR^n)$.
\end{lem}
In the following section, we show that if a nice function $u$ is a solution to the equation \eqref{DP}, then it satisfies the following weak version of the equation: $\langle u, \phi \rangle_\nu = (f,\phi)$, for every function $\phi\in \hkrn$, under the condition $u = g$ outside $\Omega$.

\section{Weak/variational solutions}\label{sec:weak}
We define the strong solutions to \eqref{DP} as the functions which satisfy its equations almost everywhere. However, our main target of consideration are the weak solutions.
\begin{defi}
Let $f\in L^2(\Omega)$. We say that $u \in V_\nu^{\Omega}(\mR^n)$ is a weak solution to \eqref{DP}, $u = g$  a.e. outside $\Omega$, and for every $\phi \in \hkrn$
\begin{equation}\label{WEAK}
\frac 12 \porn\porn (u(x) - u(y))(\phi(x) - \phi(y))d\nu_y(x) dy = \poom f(x)\phi(x) dx.
\end{equation}
In short, $u = g $ a.e. in $\Omega^c$ and 
\begin{equation}\label{short}
\langle u,\phi \rangle_\nu = (f,\phi).
\end{equation}
\end{defi}
The definition implies, that a neccessary condition for the existence of weak solution is that $g$ can be extended to a $\vkrn$ function. This also turns out to be a sufficient condition. In order to provide a more constructive assumption on $g$, one needs to consider the extension problem, which we discuss in Section \ref{sec:ext}. In there we also formulate a fully constructive set of assumptions under which the Dirichlet problem has a weak solution, see Corollary \ref{exicons}. Recall that $\Omega$ is bounded, and let us present the main result of this section.
\begin{thm}\label{exigen}
	Let $\nu$ be a symmetric L\'evy measure. If $f\in L^2(\Omega)$ and there exists $h\in\vkrn$, such that $g=h\restriction_{\Omega^c}$, then the equation of the form \eqref{WEAK} has a unique solution $u\in \vkrn$.
\end{thm}
To prove this theorem, we note that the quadratic form $\langle u,u \rangle_\nu$ can be represented in terms of forms $\langle u,u\rangle_{\delta_y}$, where $\delta_y$ is the Dirac delta at point $y$. Then we establish the Poincar\'e inequality for atomic measures, and use the aforementioned representation of $\nu$, to prove that the Poincar\'e inequality holds for every symmetric L\'evy measure in Theorem \ref{poingen}. After doing that, we use the Lax-Milgram theorem to finish the proof for the homogeneous case ($g=0$), from which we pass to the non-homogeneous case.
\begin{thm}[Lax-Milgram theorem, \cite{Lax}, \S 6 Th. 6]
	Let $\mathcal{H}$ be a Hilbert space over $\mathbb{K} = \mathbb{C}$ or $\mR$, and let $a: \mathcal{H} \times \mathcal{H} \longmapsto \mathbb{K}$ be a bilinear functional that satisfies
	\begin{enumerate}
		\item $(\exists C > 0) (\forall x,y \in \mathcal{H}) \ \  |a(x,y)| \leq C\|x\|\cdot \|y\|$,
		\item $(\exists \beta >0)(\forall x \in \mathcal{H}) \ \ |a(x,x)| \geq \beta \|x\|^2$ (coercivity).
	\end{enumerate}
	Then for every $l \in \mathcal{H}^*$ the equation
	$$a(u,v) = l(v) \ \ \ \ \ \hbox{for every } v\in \mathcal{H},$$
	has a unique solution $u$.
\end{thm}
\begin{lem}\label{delty}
	For every L\'evy measure $\nu$ and $u\in \hkrn$ we have
	\begin{equation}
	\langle u,u \rangle_\nu = \porn \langle u,u \rangle_{\delta_{y}} d\nu(y).
	\end{equation}
\end{lem}
\begin{proof}
	Let $u\in\hkrn$. By Tonelli's theorem, we have
	\begin{equation*}
	\langle u,u \rangle_\nu = \frac 12\porn\porn (u(x) - u(x+y))^2 d\nu(y)dx = \frac 12\porn\porn (u(x)- u(x+y))^2 dxd\nu(y).
	\end{equation*}
	Again, by Tonelli's theorem, we can iterate the integration to get
	\begin{equation}
	\langle u,u \rangle_{\delta_y} = \frac 12\porn\porn (u(x) - u(x+z))^2 d\delta_y(z) dx = \frac 12\porn (u(x) - u(x+y))^2 dx,
	\end{equation}
	which ends the proof.
\end{proof}
A similar formula holds with $u,v\in\hkrn$, and $\langle u,v\rangle$ instead of $\langle u,u\rangle$. 
Let us note the following fact which is a consequence of the formula $2a^2 + 2b^2 \geq (a+b)^2$.
\begin{lem}\label{atom}
	Let $B$ be a Borel set in $\mR^n$, $x_0 \in \mR^n\backslash\{0\}$. For every $u$ for which the right hand side makes sense,
	\begin{equation*}
	\int\limits_{B} (u(x) - u(x+x_0))^2 dx \geq \frac 12\int\limits_{B} u(x)^2 dx - \int\limits_{B+x_0} u(x)^2 dx.
	\end{equation*}
\end{lem}
\begin{lem}[Poincar\'e inequality for measures with atoms]\label{poi}
	Let $\nu$ be a L\'evy measure with an atom in $x_0\in \mR^n\backslash\{0\}$. Then the quadratic form $\langle \cdot, \cdot \rangle_\nu$ satisfies the Poincar\'e inequality
	\begin{equation}
	\label{poin}
	C\langle u,u \rangle_\nu \geq \|u\|^2_{L^2(\mR^n)}, \hfill \hbox{ for every } u\in\hkrn,
	\end{equation}
	with $C$ independent of $u$. Furthermore, if we fix $\Omega$ and $\eps > 0$, then for $|x_0|> \eps$ the constant is uniformly bounded. 
\end{lem}

\begin{proof}
	It suffices to consider measures of the form $\nu(A) = \delta_{x_0}(A)$ for an arbitrary $x_0\in\mR^n$. Let us write the quadratic form for such a measure
	\begin{align}
	2\langle u,u\rangle_\nu &= \label{yield} \porn\porn (u(x) - u(x+y))^2 d\nu(y) dx =  \porn (u(x) - u(x+x_0))^2 dx  \\
	&\nonumber= \int\limits_{\Omega^c - x_0} u(x)^2 dx + \int\limits_{\Omega - x_0} (u(x) - u(x+x_0))^2 dx\\
	&=\label{delta} \int\limits_{(\Omega^c - x_0)\cap\Omega} u(x)^2 dx + \int\limits_{\Omega - x_0}(u(x) - u(x+x_0))^2 dx.
	\end{align}
	By \eqref{yield} and \eqref{delta} we see that it is enough to show that $\widetilde{C} \langle u,u\rangle_\nu \geq \int\limits_{(\Omega-x_0)\cap \Omega} u(x)^2 dx$ with $\widetilde{C}$ independent of $u$.\\
	Using Lemma \ref{atom}, we get
	\begin{equation}\label{1krok}
	2\langle u,u\rangle_\nu \geq \int\limits_{(\Omega-x_0)\cap \Omega} (u(x) - u(x+x_0))^2 dx \geq \frac 12 \int\limits_{(\Omega-x_0) \cap \Omega} u(x)^2 dx - \int\limits_{\Omega\cap (\Omega + x_0)}u(x)^2 dx.
	\end{equation}
		Again, by Lemma \ref{atom} and the fact that $u$ is supported in $\Omega$: 
		\begin{equation}\label{krok2}
		4\langle u,u\rangle_\nu \geq 2\int\limits_{\Omega\cap(\Omega+x_0)} (u(x) - u(x+x_0))^2 dx \geq \int\limits_{\Omega\cap (\Omega + x_0)}u(x)^2 dx - 2\int\limits_{\Omega\cap(\Omega + x_0)\cap(\Omega + 2x_0)} u(x)^2 dx.
	\end{equation}
	Adding \eqref{1krok} and \eqref{krok2} side by side yields
	\begin{equation}\label{comp2}
	6\langle u, u \rangle_\nu \geq \frac 12 \int\limits_{\Omega\cap(\Omega - x_0)} u(x)^2dx - 2\int\limits_{\Omega\cap(\Omega + x_0)\cap(\Omega + 2x_0)} u(x)^2 dx.
	\end{equation}
%	Hence, looking back at \eqref{yield} and \eqref{delta}, we see that
%	\begin{align}
%	2\langle u,u \rangle_\nu &= \int\limits_{(\Omega^c - x_0)\cap\Omega} u(x)^2 dx +\int\limits_{\Omega - x_0} (u(x) - u(x+x_0))^2 dx\nonumber\\
%	&\geq \int\limits_{\Omega^c - x_0} u(x)^2 dx+ \frac 12 \int\limits_{\Omega\cap(\Omega - x_0)} u(x)^2dx - \int\limits_{\Omega\cap(\Omega + x_0)} u(x)^2 dx. \label{compen}
%	\end{align}
%	\eqref{compen} is almost the desired expression. We only need to deal with the subtracted term. We compensate it by adding $4\langle u,u \rangle_\nu$ to both sides of the inequality:
%	\begin{equation}\label{compen2}
%	6\langle u,u \rangle_\nu \geq \int\limits_{\Omega^c - x_0} u(x)^2 dx+ \frac 12 \int\limits_{\Omega\cap(\Omega - x_0)} u(x)^2dx - \int\limits_{\Omega\cap(\Omega + x_0)} u(x)^2 dx + 4 \langle u,u \rangle_\nu.
%	\end{equation}
%%	Again, by Lemma \ref{atom} and the fact that $u$ is supported in $\Omega$: 
%%	\begin{equation}\label{krok2}
%%	4\langle u,u\rangle_\nu \geq 2\int\limits_{\Omega\cap(\Omega+x_0)} (u(x) - u(x+x_0))^2 dx \geq \int\limits_{\Omega\cap (\Omega + x_0)}u(x)^2 dx - 2\int\limits_{\Omega\cap(\Omega + x_0)\cap(\Omega + 2x_0)} u(x)^2 dx.
%%	\end{equation}
%	Plugging \eqref{krok2} into \eqref{compen2} yields
%	\begin{equation}\label{comp2}
%	6\langle u, u \rangle_\nu \geq \int\limits_{\Omega^c - x_0} u(x)^2 dx+ \frac 12 \int\limits_{\Omega\cap(\Omega - x_0)} u(x)^2dx - 2\int\limits_{\Omega\cap(\Omega + x_0)\cap(\Omega + 2x_0)} u(x)^2 dx.
%	\end{equation}
	In the next step we use Lemma \ref{atom} with $8\langle u,u\rangle_\nu\geq 4\int\limits_{\Omega\cap(\Omega + x_0)\cap(\Omega + 2x_0)}(u(x) - u(x+x_0)^2 dx$.\\
	At every step we obtain an inequality of the form:
	\begin{equation}\label{kty}
	C_k \langle u,u\rangle_\nu \geq \frac 12 \int\limits_{\Omega\cap(\Omega - x_0)}u(x)^2 dx -c_k \int\limits_{\Omega\cap (\Omega + x_0) \cap \ldots\cap (\Omega + kx_0)} u(x)^2 dx.
	\end{equation} However, since $\Omega$ is bounded, for some $n\in\mathbb{N}$, we will get $\Omega\cap(\Omega+x_0)\cap(\Omega+2x_0)\cap\ldots\cap(\Omega + nx_0) = \emptyset$. Then the subtracted integral in \eqref{kty} is equal to 0, and we get the desired result.\\
	The uniform boundedness of $C$ follows directly from the proof: notice how the ratio of $diam(\Omega)$ to $|x_0|$ affects the required number of steps in our reasoning.
\end{proof}
\begin{thm}[Poincar\'e inequality for symmetric L\'evy measures]\label{poingen}
Let $\Omega$ be a nonempty bounded open set, and let $\nu$ be a symmetric L\'evy measure. Then,
\begin{equation}\label{coer}
		\|u\|_{L^2(\Omega)}^2 \leq C(\nu,\Omega) \langle u,u \rangle_\nu  \hbox{\ \ \ \ \ for every } u\in \hkrn.
	\end{equation}
\end{thm}
\begin{proof}
	Let $a > b > 0$, and let $R_a^b = \{x\in \mR^n: a\leq |x| \leq b\}$. Note that for every L\'evy measure $\nu$ there exist $\eps_2 > \eps_1 > 0$ such that $\nu(R_{\eps_1}^{\eps_2}) > 0$. By Lemma \ref{poi}, there exists $C>0$, such that for every $y\in R_{\eps_1}^{\eps_2}$ and $u\in \hkrn$
	\begin{equation*}
		\langle u,u \rangle_{\delta_{y}} \geq  C^{-1}\|u\|_{L^2(\Omega)}^2.
	\end{equation*}
	Hence,
	\begin{align}
		\langle u,u \rangle_\nu &= \frac 12 \porn \porn (u(x) - u(x+y))^2 dx d\nu(y) \geq \int\limits_{R_{\eps_1}^{\eps_2}} \frac 12\porn (u(x) - u(x+y))^2 dx d\nu(y) \nonumber\\
		&= \int\limits_{R_{\eps_1}^{\eps_2}} \langle u,u \rangle_{\delta_y} d\nu(y) \geq C^{-1} \nu(R_{\eps_1}^{\eps_2})\|u\|_{L^2(\Omega)}^2.\label{Poincare}
	\end{align}
\end{proof}
\begin{proof}[Proof of Theorem \ref{exigen}]
	First we take care of the homogeneous equation, i.e. $g=0$ a.e. outside $\Omega$.
	Let us use the Lax-Milgram theorem with $\mathcal{H} = \hkrn$ with the norm $\|\cdot\|_{\vkrn}$, $a(u,v) = \langle u, v \rangle_\nu$, and $l(v) = (f,v)$. For $u,v\in \hkrn$ we have
	$$|\langle u,v \rangle_\nu|^2 \leq \langle u,u \rangle_\nu \langle v,v \rangle_\nu \leq (\|u\|_2^2 + \langle u,u \rangle_\nu)(\|v\|_2^2 + \langle v,v \rangle_\nu) \leq 2\|u\|_{V_\nu^{\Omega}(\mR^n)}^2 2\|v\|_{V_\nu^{\Omega}(\mR^n)}^2,$$
	hence $a$ is bounded.\\ In our setting the coercivity is equivalent to
	$$\langle u,u \rangle_\nu \geq \beta (\|u\|_{L^2(\Omega)}^2 + \langle u,u \rangle_\nu) \hbox{\ \ \ \ \ for every } u\in \hkrn.$$
	As we see, $\beta$ must be a number from the interval $(0,1)$. Thus, the coercivity is granted by Theorem \ref{poingen}.
	
	For every $\phi\in\hkrn$,
	$$|(f,\phi)| \leq \|f\|_{L^2(\Omega)} \|\phi\|_{L^2(\Omega)} \leq \|f\|_{L^2(\Omega)} \|\phi\|_{V_\nu^{\Omega}(\mR^n)},$$
	hence $l\in\mathcal{H}^*$. \\
	By the Lax-Milgram theorem we conclude that the equation
	$$\langle u, \phi \rangle_\nu = (f,\phi) \ \ \ \ \ \hbox{for every } \phi\in\hkrn,$$
	has a unique solution $u \in \hkrn$.\\
	The case of $g\not\equiv 0$ can now be resolved quite easily.\\
	Consider an arbitrary (fixed) extension of $g$ to a function in $\vkrn$ (which we also call $g$). Note that the conditions ``$u\in\vkrn$, $u=g$ a.e. in $\Omega^c$'' are equivalent to ``$u = \wtu + g$ for some $\wtu \in\hkrn$''. Let $u = \wtu + g$ be such a function. Then
	\begin{align*} &\porn\porn (u(x) - u(y))(\phi(x) - \phi(y))d\nu_y(x) dy
		\\ &= \porn\porn (\wtu(x) + g(x) - \wtu(y) - g(y))(\phi(x) - \phi(y))d\nu_y(x) dy
		\\ &= \porn\porn (\wtu(x) - \wtu(y))(\phi(x) - \phi(y)) d\nu_y(x) dy 
		\\ &\ \ \ \ \ +\porn\porn (g(x) - g(y))(\phi(x) - \phi(y)) d\nu_y(x) dy.\end{align*}
	Since $\wtu\in\hkrn$, the existence of the solution of the equation \eqref{WEAK} is equivalent to the existence of the solution $\wtu$ of the homogeneous equation
	\begin{equation}\label{nh}
		\langle \wtu, \phi \rangle_\nu = (f,\phi) - \langle g, \phi \rangle_\nu \ \ \ \ \ \hbox{for every }\phi\in\hkrn.
	\end{equation}
	By the Cauchy-Schwarz inequality and the fact that $g\in\vkrn$, we have
	$$\langle g, \phi \rangle_\nu^2 \leq \langle g,g\rangle_\nu \langle \phi,\phi\rangle_\nu \leq 2\langle g,g\rangle_\nu \|\phi\|_{\vkrn}^2 \ \ \ \ \hbox{for every } \phi \in \hkrn,$$
	hence $l(\cdot) = (f,\cdot) - \langle g,\cdot \rangle_\nu$ is a continuous linear functional on $\hkrn$. Thus we conclude that the equation \eqref{nh} has a unique solution $\wtu$. Therefore $u = \wtu + g$ solves \eqref{WEAK}. We claim that $u$ does not depend on the choice of the extension of $g$. Let $g_1, g_2 \in \vkrn$ be extensions of $g$, and let $\wtu, \overline{u}$ be solutions of \eqref{nh} with $g=g_1$, $g=g_2$, respectively. For every $\phi \in \hkrn$ we have
	$$\langle \wtu, \phi \rangle_\nu = (f,\phi) - \langle g_1, \phi \rangle_\nu,$$
	$$\langle \overline{u},\phi\rangle_\nu = (f,\phi) - \langle g_2, \phi \rangle_\nu.$$
	Therefore
	$$\langle \wtu + g_1 - (\overline{u} + g_2), \phi\rangle_\nu = 0 \ \ \ \ \ \hbox{for every } \phi\in\hkrn.$$
	In particular,
	$$\langle \wtu + g_1 - (\overline{u} + g_2),\wtu + g_1 - (\overline{u} + g_2) \rangle_\nu = 0.$$
	By the coercivity of $\langle \cdot, \cdot \rangle_\nu$ on $\hkrn$, we get $\wtu + g_1 = \overline{u} + g_2$ a.e. on $\mR^n$, as claimed. That proves the uniqueness of the solution.
\end{proof}
\begin{rem}
	The Poincar\'e inequality is well-known for the transient L\'evy processes, see e.g. \cite{FITZSIMMONS2000548} (1.18), or \cite{Fukushima} Theorem 2.4.2 with $d\mu = \textbf{1}_\Omega dx$. Note that not every process with generator given by \eqref{PV} is transient, see e.g. Example 35.7 in \cite{Sato}.
\end{rem}

\noindent In the sequel we shall explain why the definition of weak solutions is appropriate.\\
Let us recall the proof of the fact that being a weak solution is equivalent to being a variational solution i.e., minimizing a certain energy functional.
\begin{lem}\label{vari}
A function $u\in\vkrn$ is a solution to \eqref{WEAK} if and only if $u=g$ a.e. in $\Omega^c$ and $u$ minimizes the energy functional
\begin{equation}
E(u) = \frac 14 \bezomkw (u(x) - u(y))^2 d\nu_x(y) dx - \poom fu
\end{equation}
among the functions equal almost everywhere to $g$ on $\Omega^c$.
\end{lem}
\begin{proof}
Let $u\in V_g = \vkrn\cap \{h: h=g \hbox{ a.e. in }\Omega^c \}$ minimize $E$ among the functions from $V_g$. Then, for every $\phi \in\hkrn$ and every $\lambda \in\mR$, we have $u + \lambda \phi \in V_g$, hence
\begin{align*}
0 \leq E(u+\lambda \phi) - E(u) = \lambda \left(\langle u,\phi \rangle_\nu  - \poom f\phi\right) + \frac 12 \lambda^2 \langle \phi, \phi \rangle_\nu.
\end{align*}
For $\lambda > 0$, dividing both sides by $\lambda$ and taking the limit $\lambda \to 0^+$ gives
\begin{equation}\label{geq}
\langle u,\phi \rangle_\nu - \poom f\phi \geq 0.
\end{equation}
The same procedure for $\lambda < 0$ yields
\begin{equation}\label{leq}
\langle u, \phi \rangle_\nu - \poom f\phi \leq 0.
\end{equation}
By \eqref{geq} and \eqref{leq}, $u$ is a weak solution.\\
Now, assume that $u$ is a weak solution. Note that $V_g = u + \hkrn$, thus it suffices to check that $E(u + \phi) - E(u) \geq 0$ for every $\phi \in \hkrn$. In fact, since $u$ is a weak solution,
\begin{align*}
E(u+\phi) - E(u) = \langle u,\phi \rangle_\nu - \poom f\phi + \langle \phi, \phi \rangle_\nu = \langle \phi, \phi \rangle_\nu \geq 0.
\end{align*}

\end{proof}
We will show that if a function $u$ satisfying \eqref{DP} is sufficiently regular, then it is a weak solution.
\begin{lem}\label{smthhkrn}
If $u$ is locally Lipschitz and bounded, then $u\in\vkrn$.
\end{lem}
\begin{proof}
	Note that there exists $C>0$ such that for every $x\in\Omega$ we have $(u(x) - u(x+y))^2 \leq C(1\wedge |y|^2)$. Indeed, when $|y|\leq 1$ the inequality follows from the Lipschitz condition and the boundedness of $\Omega$, while for $|y|>1$ we use the boundedness of $u$. Therefore
	\begin{equation*}
	\poom\porn (u(x) - u(x+y))^2 d\nu(y) dx \leq C\poom\porn (1\wedge |y|^2) d\nu(y)dx = C|\Omega|\porn (1\wedge|y|^2) d\nu(y) < \infty.
	\end{equation*}
	The statement follows from Lemma \ref{sym}.
\end{proof}
%The following arguments are well-known, cf. \cite{Fukushima}, \S 1.2.
%\begin{proof}
%Let $u\in Lip_c(\mR^n)$, $u\not\equiv 0$, $F = \hbox{supp}(u)$, $\delta > 0$, and $F_{\delta} = \{x\in\mR^n: \ \dist(x,F) < \delta\}$. We have
%\begin{align*}
%&\porn\porn (u(x+y) - u(y))^2 d\nu(x) dy \\
%&= \int\limits_{F_{\delta}} \porn (u(x+y) - u(y))^2 d\nu(x) dy 
%+ \int\limits_{F_{\delta}^c} \porn u(x+y)^2 d\nu(x) dy.
%\end{align*}
%Since $(u(x+y) - u(y))^2 \leq C(|x|^2 \wedge 1)$ for some $C>0$, we have
%\begin{align*} &\int\limits_{F_{\delta}} \porn (u(x+y) - u(y))^2 d\nu(x) dy \leq \int\limits_{F_{\delta}}C \porn (1 \wedge |x|^2)d\nu(x) dy\\
%&\leq C|F_{\delta}|\porn (1\wedge |x|^2) d\nu(x) < \infty.
%\end{align*}
%By the definition of $F_{\delta}$, we have $(F_{\delta}^c + x) \cap F = \emptyset$ for every $|x| < \delta$. Hence, by Fubini's theorem,
%\begin{align*}
%&\int\limits_{F_{\delta}^c} \porn u(x+y)^2 d\nu(x) dy = \porn\int\limits_{F_{\delta}^c + x} u(y)^2 dy d\nu(x) =\int\limits_{B(0,\delta)^c} \int\limits_{F_{\delta}^c + x}u(y)^2 dy d\nu(x) 
%\\ &\leq \nu(B(0,\delta)^c) \|u\|_{L^2(\mR^n)}^2 <\infty.
%\end{align*}
%Since $u\in L^2(\mR^n)$, we have $\|u\|_{H_\nu(\mR^n)} < \infty$, as we wanted to prove.
%\end{proof}
\begin{thm}\label{strwk}
If $u\in C^2_b(\mR^n)$ is a solution to \eqref{DP}, then it is also a weak solution.
\end{thm}
\begin{proof}
Assume that $u\in C^2_b(\mR^n)$ and let $\varepsilon >0$. By Proposition \ref{smthint}, $Lu(x) = \frac 12\porn (2u(x) - u(x+y) - u(x-y)) d\nu(y)$ converges absolutely and $f:= Lu \in L^2(\Omega)$. If $\phi\in\hkrn$, then by Tonelli's theorem
\begin{align*}
&\int\limits_{\Omega\times \mR^n} \frac 12\left|\phi(x)(2u(x) - u(x+y) - u(x-y))\right| d\nu(y) dx 
\\ &\leq C\int\limits_{\Omega \times \mR^n} |\phi(x)|(1\wedge |y|^2) d\nu(y) dx \nonumber
\\ &= C\poom |\phi(x)| \porn (1 \wedge |y|^2) d\nu(y) dx < \infty. \nonumber
\end{align*} 
By this and the dominated convergence theorem, for every $\phi\in\hkrn$ we have 
\begin{align}
\poom f \phi &= \poom Lu(x) \phi(x) dx \nonumber
\\ &= \poom\porn \frac 12 \phi(x)  (2u(x) - u(x+y) - u(x-y)) d\nu(y) dx \nonumber
\\ \label{przspl} &= \lim\limits_{\varepsilon \to 0^+} \poom\phi(x) \int\limits_{B(0,\varepsilon)^c}\frac 12 (2u(x) - u(x+y) - u(x-y)) d\nu(y) dx 
\\ \label{pospl} &= \lim\limits_{\varepsilon \to 0^+} \poom\phi(x) \int\limits_{B(0,\varepsilon)^c}(u(x) - u(x+y))d\nu(y) dx
\\ &= \lim\limits_{\varepsilon \to 0^+} \porn\phi(x) \int\limits_{B(0,\varepsilon)^c}(u(x) - u(x+y))d\nu(y) dx. \nonumber
\end{align}
Splitting the integral in \eqref{przspl} is legitimate, since the integral over $B_0(\varepsilon)^c$ in \eqref{pospl} is bounded as a function of $x$. This was shown in the proof of Proposition \ref{smthint}. By the symmetry of $\nu$ and $B(0,\varepsilon)$, and translation invariance of Lebesgue measure, we have
\begin{align*}
&\porn \phi(x) \int\limits_{B(0,\varepsilon)^c} (u(x) - u(x+y)) d\nu(y) dx
\\&= \int\limits_{B(0,\varepsilon)^c} \porn \phi(x) (u(x) - u(x+y)) dx d\nu(y)
\\&= \int\limits_{B(0,\varepsilon)^c} \porn \phi(x-y)(u(x-y) - u(x)) dx d\nu(y)
\\&= \int\limits_{B(0,\varepsilon)^c} \porn \phi(x+y)(u(x+y) - u(x)) dx d\nu(y)
\\&= -\porn \int\limits_{B(0,\varepsilon)^c}  \phi(x+y)(u(x) - u(x+y)) d\nu(y) dx.
\end{align*}
Therefore
\begin{align*}
&\lim\limits_{\varepsilon \to 0^+} \porn\phi(x) \int\limits_{B(0,\varepsilon)^c}(u(x) - u(x+y))d\nu(y) dx
\\&= \frac 12 \lim\limits_{\varepsilon \to 0^+} \porn \int\limits_{B(0,\varepsilon)^c} (u(x) - u(x+y))(\phi(x) - \phi(x+y))d\nu(y) dx
\\ &= \frac 12 \porn\porn (u(x) - u(x+y))(\phi(x) - \phi(x+y))d\nu(y) dx. 
\end{align*}
The last equality follows from Lemma \ref{smthhkrn}, which yields the absolute convergence of the last integral, and from the dominated convergence theorem.\\
\end{proof}

In the strong case, it is obvious that the solutions are stable under taking subspaces, i.e. if $Lu = f$ in $\Omega$, then $Lu = f$ in $\Omega' \subseteq \Omega$. A similar fact is true for weak solutions.
\begin{prop}
Let $\Omega' \subseteq \Omega$, $f \in L^2(\Omega)$, $u\in V_\nu^{\Omega}(\mR^n)$, and let $\langle u, \phi \rangle_\nu = (f,\phi)_{\Omega}$ for every $\phi \in H_\nu^{\Omega}(\mR^n)$. Then  $\langle u, \psi \rangle_\nu = (f,\psi)_{\Omega'}$ for every $\psi \in H_\nu^{\Omega'}(\mR^n)$.
\end{prop}
\begin{proof}
Note that $u\in V_\nu^{\Omega'}(\mR^n)$, and $H_\nu^{\Omega'}(\mR^n)\subseteq \hkrn$. Therefore, for every $\psi\in H_\nu^{\Omega'}(\mR^n)$ we have
\begin{equation*}
\langle u,\psi \rangle_\nu = (u,\psi)_{\Omega} = (u,\psi)_{\Omega'},
\end{equation*}
i.e. $u$ is a weak solution in $\Omega'$.
%Let $u$ satisfy our assumptions. By Theorem \ref{exigen}, there exists a unique $v\in V_\nu^{\Omega'}(\mR^n)$ for which $\langle v, \psi \rangle_\nu = (f,\psi)_{\Omega'}$ and $v = u$ a.e. outside $\Omega'$. We will show that $v = u$ a.e. in $\Omega'$. For every $\psi \in H_\nu^{\Omega'}(\mR^n)$, we have
%\begin{equation*}
%\langle u-v, \psi \rangle_\nu = \langle u, \psi \rangle_\nu - \langle v, \psi \rangle_\nu = (f,\psi)_{L^2(\Omega')} - (f,\psi)_{L^2(\Omega')} = 0.
%\end{equation*}
%In particular, since $u-v = 0$ outside $\Omega'$, we have $\langle u-v, u-v \rangle_\nu = 0$. Therefore, by the Poincar\'e inequality \eqref{Poincare}, $u-v = 0$ a.e. in $\Omega'$.
\end{proof}

\section{Maximum principle and its applications}\label{sec:comp}
\subsection{Comparison principle}
We present the so-called maximum and comparison principle for the nonlocal operator $L$. Analogous results were given for the fractional Laplacian in \cite{servadei2014}, see also the discussion in \cite{RosOton}.
\begin{thm}[Weak maximum principle]\label{wkmax}
Let $u$ be a weak solution to \eqref{DP} with $f \geq 0$, $g\geq 0$ a.e. Then $u \geq 0$ a.e.
\end{thm} 
\begin{proof}
We want to use $u_- = - (u\wedge 0)$ as the test function $\phi$ in \eqref{WEAK}. We claim that it is in $\hkrn$. Indeed, we have $g\geq 0$, hence $u_- = 0$ outside $\Omega$. Of course $u_- \in L^2(\mR^n)$. The integrability condition from Definition \ref{hkomega} follows from $(u_-(x) - u_-(y))^2 \leq (u(x) - u(y))^2$. This verifies the claim.\\
Since $u$ is a weak solution, by Lemma \ref{delty} and the fact that for any function $u$, $(u_+(x) - u_+(y))(u_-(x) - u_-(y)) \leq 0$, we get
\begin{align}
&0 \leq \poom f(x)u_-(x)dx = \langle u,u_-\rangle_\nu =  \langle u_+,u_-\rangle_\nu - \langle u_-, u_- \rangle_\nu \leq -\langle u_-, u_- \rangle_\nu. \nonumber
\end{align}
Since we also have $\langle u_-, u_- \rangle_\nu \geq 0$, we see that $\langle u_-, u_- \rangle_\nu = 0$. By the Poincar\'e inequality \eqref{Poincare} (which we can use, because $u_-\in \hkrn$) we conclude that $u_- = 0$ a.e. in $\Omega$.
\end{proof}
\begin{cor}[Weak comparison principle]\label{comp}
If $u,v$ solve \eqref{WEAK} with $f = f_u, g = g_u$ $f=f_v$, and $g=g_v$ respectively, and if $f_u \geq f_v$, $g_u \geq g_v$ , then $u \geq v$.
\end{cor}
\begin{proof}
Take $u-v$ in the theorem above.
\end{proof}
Let us reformulate Theorem \ref{wkmax} for the (strong) solutions of \eqref{DP}, to justify calling it \textit{maximum} (or rather \textit{minimum}) \textit{principle}.
\noindent In the following theorems we do not make any assumptions on $\nu$ apart from those in \eqref{Levy}. Recall that $D(L,\Omega)$ contains functions $u$, for which $Lu(x)$ exists for every $x\in\Omega$.
\begin{thm}\label{strmax}
If $u\in D(L,\Omega) \cap C(\mR^n)$ satisfies $Lu \geq 0$ in $\Omega$, and $u \geq 0$ outside $\Omega$, then $u \geq 0$ a.e. in $\Omega$.
\end{thm}
\begin{proof}
Assume by contradiction that $u(y) < 0$ for some $y\in\Omega$. Then, by continuity we conclude that $u$ has a global minimum at some $x\in\Omega$ . Since $u(x)$ is the global minimum of $u$, we have $u(x) - u(x+y) \leq 0$ for every $y\in\mR^n$. Therefore, by the monotone convergence theorem, we can drop the PV in \eqref{PV} getting $Lu(x) = \porn (u(x) - u(y))d\nu_x(y) = \porn (u(x) - u(x+y))d\nu(y) \leq 0$. If $\porn (u(x) - u(y)) d\nu_x(y) < 0$, then we get the desired contradiction. Otherwise, let $A\subset \mR^n$ be such that $\nu(A) > 0$, $\dist(0,A) > 0$. In addition, we want $A+x$ to dominate $x = (x_1,\ldots,x_n)$ on at least one coordinate, i.e. for some $k \in \{1,\ldots,n\}$ and every $y\in A+x$ we have $y_k - x_k \geq d  >  0$. Since $\porn (u(x) - u(y))d\nu_x(y) \leq \int\limits_{A+x} (u(x) - u(y))d\nu_x(y) = 0$, we get that $u(y) = u(x) < 0$ $\nu_x$-a.e. on $A+x$. Let $x_1 \in A + x$ be such that $u(x_1) = u(x)$. We have $(x_1)_k \geq x_k + d$. Once again, if $Lu(x_1) < 0$, then we have a contradiction, and if $Lu(x_1) = 0$ we repeat the procedure obtaining $x_2$, and so on. Since $A$ dominates $0$ and $\Omega$ is bounded, we will eventually get that for some $m$ either $Lu(x_m) < 0$ or $x_m \in \Omega^c$ and $u(x_m) = u(x) < 0$ which contradicts $u(y) > 0$ for $y \in \Omega^c$. 

\end{proof}
The first iteration of the argument above gives the proof of the negative minimum (equivalently - positive maximum) principle.
\begin{prop}\label{negmax}
If $u\in C(\mR^n)$ satisfying $u\geq 0$ outside $\Omega$ has a negative global minimum at $x\in \Omega$ and $u\in D(L,\Omega)$, then $Lu(x) \leq 0$. If the minimum is strict, then $Lu(x) < 0$.
\end{prop}
\begin{exa}
Without the assumption that the maximum at $x$ is strict, $Lu(x)$ is not necessarily strictly negative. Consider the L\'evy measure $\delta_1 + \delta_{-1}$ on $\mR$, let $\Omega = (-2,2)$ and let $u\in C_c^\infty(\mR)$ satisfy $0\geq u\geq -1$, $u(x)=0$ for $|x|>2$, $u(x)=-1$ for $|x|<3/2$. Clearly $Lu(0) = 0$. 
\end{exa}
By looking at the last iteration in the proof of Theorem \ref{strmax}, we can refine Proposition \ref{negmax}.
\begin{prop}
	If $u\in C(\mR^n)$ satisfying $u\geq 0$ outside $\Omega$ has a negative global minimum at $x\in \Omega$ and $u\in D(L,\Omega)$, then there exists $x'\in \Omega$ such that $u(x') = u(x)$, and $Lu(x') < 0$.
\end{prop}
\subsection{Barriers}
Let us construct barriers, i.e. compactly supported functions, smooth in $\Omega$, satisfying
$$\begin{cases}
Lw \geq 1 \hbox{ in } \Omega,\\
w \geq 0,\\
w \leq C \hbox{ in } \Omega,
\end{cases}$$
with $C$ depending on $\nu$ and $\Omega$.

Taking our cue from the work of Ros-Oton \cite{RosOton}, we use different approaches depending on whether $\nu$ is compactly supported or not.
\subsubsection{Barrier for $\nu$ with unbounded support}
Consider $R>0$ so large that $\overline{\Omega} \subset B_R$, and $\eta \in C_c^{\infty}(B_R)$ such that $0\leq \eta \leq 1$, for $x\in \mR^n$, and $\eta(x) = 1$ for $x\in\Omega$. Then $\eta(x) - \eta(x+y) \geq 0$ for $x\in \Omega, y \in \mR^n$. Thus we can drop the PV in \eqref{DP} when we compute $L\eta$ for $x\in\Omega$, and
\begin{align*} L\eta(x) &= \porn(\eta (x) - \eta (x+y)) d\nu(y) = \porn (\eta(x) - \eta(y))d\nu_x(y) \\
&\geq \int\limits_{B_R^c} d\nu_x(y) \geq \int\limits_{B_{2R}^c} d\nu(y) \geq C > 0.
\end{align*}
Function $w(x) = \frac 1C \eta (x)$ satisfies the desired conditions.\\
By Proposition \ref{smthint}, and Theorem \ref{strwk}, we know that the above barrier is also a weak solution with $f:= Lw\in L^2(\Omega)$.
\subsubsection{Barrier for compactly supported $\nu$}
Consider a sufficiently large $r_1$ so that $\nu(B_{r_1}^c) = 0$, and let $r_2  = \sup\{|x|: x \in \Omega\}$ . For $R = r_1 + r_2 + 1$ and $x\in\mR^n$, we set $\eta(x) = ((1- \frac {|x|^2}{R^2})\vee 0)$. Inside $B_R$, $\eta$ is smooth and strictly concave. In particular, for $x\in\Omega$ for every $\eps > 0$, there exists $\widetilde{C}>0$ such that if $\eps < |y| < r_1$, then $2\eta(x) - \eta(x+y) - \eta(x-y) \geq \widetilde{C}$. By the smoothness of $\eta$ in $B_R$, and the choice of $R$, $L\eta(x)$ is well defined, and $L\eta \in L^\infty (\Omega)$:
	\begin{equation*}
	\porn (2\eta(x) - \eta(x+y) - \eta(x-y)) d\nu(y) \leq C'\porn (1\wedge|y|^2)d\nu(y) < C''. 
	\end{equation*}
On the other hand, for every $x\in\Omega$, we have
	\begin{equation*}
	L\eta(x) = \porn (2\eta(x) - \eta(x+y) - \eta(x-y)) d\nu(y) \geq \int\limits_{B_{r_1}\backslash B_\eps} (2\eta(x) - \eta(x+y) - \eta(x-y)) d\nu(y) \geq \widetilde{C} \nu(B_{r_1}\backslash B_\eps).
	\end{equation*}
Hence, the function $w(x) = \frac 1{\widetilde{C}\nu(B_{r_1}\backslash B_\eps)} \eta(x)$ is our desired barrier. Note that $w\in \vkrn$. Indeed, $w$ is Lipschitz in $B_{R-1}$, hence we have
	\begin{equation*}
	\poom\porn (w(x) - w(x+y))^2 d\nu(y) dx \leq \overline{C}\poom \porn (1\wedge|y|^2) d\nu(y) dx < \infty.
	\end{equation*}
Furthermore, all calculations from the proof of Lemma \ref{strwk} are correct if we put $w$ instead of $u$. Hence $w$ is a weak solution with $f_w := Lw\in L^2(\Omega)$.

Note that the function $\eta$ from the unbounded case could fail when $supp (\nu) \subseteq B_R$: if $x\in\Omega$, $d(x,\Omega^c) > R$, then $L\eta(x) = 0$ because $\eta\equiv 1$ in $\Omega$. On the other hand, $\eta$ from bounded case is not concave on the whole of $\mR^n$, hence $2u(x) - u(x+y) - u(x-y) \geq \widetilde{C}$ might not hold for large $y$.

Now we will use the barriers to obtain $L^\infty$ bounds for solutions.
\begin{lem}\label{barr}
Let u be a solution to \eqref{WEAK}. Then there exists a constant $c$ independent of $f$ and $g$, such that
\begin{equation}\label{esti}
\|u\|_{\esom} \leq c\|f\|_{\esom} + \|g\|_{L^{\infty}(\mR^n \backslash \Omega)}.
\end{equation}
\end{lem}
\begin{proof}
We may assume that $f$ and $g$ are bounded. Define $v(x) = \|f\|_{\esom}\cdot w(x) + \|g\|_{L^{\infty}(\mR^n \backslash \Omega)}$, where $w$ is the appropriate barrier. Obviously, $v \geq u$ outside $\Omega$. We have $Lv(x) = \|f\|_{\esom}\cdot Lw(x) =: f_v(x)$ for $x\in\Omega$. Since $w$ is a weak solution, we get that $\langle v, \phi \rangle_\nu = (f_v,\phi)$ for every $\phi\in\hkrn$. Since $Lw \geq 1$, we have $f_v \geq f$. Therefore, by Corollary \ref{comp}, $v \geq u$. Since $w \leq C$ in $\mR^n$, we see that
$$u \leq C\|f\|_{\esom} + \|g\|_{L^{\infty}(\mR^n \backslash \Omega)}.$$
A similar argument using $-v$ shows that
$$ u \geq -(C\|f\|_{\esom} + \|g\|_{L^{\infty}(\mR^n \backslash \Omega)}).$$
This completes the proof.
\end{proof}
The method of barriers works just as well for the strong solutions, as long as they enjoy the comparison principle (cf. Theorem \ref{strmax}).
\begin{exa}
We will use the barrier to estimate the solution to the equation \eqref{exit}. We have
\begin{equation}\label{exbarr}
\begin{cases}
Ls = 1, & \hbox{in }\Omega,\\
s = 0, & \hbox{in }\Omega^c.
\end{cases}
\end{equation}
By Lemma \ref{barr}, for some $C>0$ and all $x\in\Omega$, we get $s(x) \leq C$. In particular the mean first exit time from a nonempty bounded open set for a jump L\'{e}vy process is finite if the intensity of jumps is positive. See \cite{Pruitt}, \cite{BGR}, and \cite{BJ} for the probabilistic approach.
\end{exa}

In the sequel, we construct more effective barriers for the unbounded case, in order to enhance the constant in \eqref{esti}.
\begin{thm}\label{effbarr} If $u$ is a solution for $\eqref{WEAK}$ with $\nu$ having unbounded support, then
$$u \leq C\|f\|_{\esom} + \|g\|_{L^{\infty} (\mR^n \backslash \Omega)}$$
and $C^{-1} = \lim\limits_{\varepsilon\to 0^+}\inf\limits_{x\in\Omega} \nu(\Omega_{\varepsilon} - x)$.
\end{thm}
\begin{proof}
Let $\varepsilon > 0$ and $\Omega_{\varepsilon} = \{x\in\Omega: \dist(x,\Omega) < \varepsilon \}$. Let us consider $\eta_{\varepsilon}\in C_c^{\infty}(\Omega_{\varepsilon})$ such that $0 \leq \eta_{\varepsilon} \leq 1$, and $\eta_{\varepsilon} = 1$ in $\Omega$. For $x\in\Omega$ we have
$$L\eta_{\varepsilon}(x) = \porn (\eta_{\eps}(x) - \eta_{\eps}(y)) d\nu_x(y) \geq \int\limits_{\Omega_{\varepsilon}^c} d\nu_x(y) =: \kappa^{\Omega_\varepsilon}(x).$$
In particular, for every $x\in\Omega$, we get $L\eta_{\varepsilon}(x) \geq \inf\limits_{x\in \Omega} \kappa^{\Omega_\varepsilon}(x) =: C_{\varepsilon}^{-1}$, thanks to which we obtain a barrier $w_{\varepsilon}$ with $Lw_{\varepsilon} \geq 1$, $0 \leq w_{\varepsilon} \leq C_{\varepsilon}$. Repeating the proof of Lemma \ref{barr} yields $u \leq C_{\varepsilon}\|f\|_{\esom} + \|g\|_{L^{\infty} (\mR^n \backslash \Omega)}$ for every $\varepsilon > 0$. Since $C_\varepsilon$ is increasing and bounded from above by $\inf\limits_{x\in\Omega} \nu(\Omega - x)$, we obtain
$$u \leq \lim\limits_{\varepsilon \to 0^+}C_{\varepsilon}\|f\|_{\esom} + \|g\|_{L^{\infty} (\mR^n \backslash \Omega)}.$$
\end{proof}

\noindent In \cite{BJ}, Bogdan and Jakubowski give a slightly better estimate 
\begin{equation}\label{BJ}
C^{-1} =\inf\limits_{x\in\Omega} \kappa^{\Omega}(x)
\end{equation}
under different assumptions. The following example shows that the expression above is not always equal to our estimate.
\begin{exa}
Let $n=1$ and $\nu = \sum\limits_{k\in \mathbb{Z}\backslash \{0\}}  \delta_k \frac 1{k^2}$, where $\delta_x$ is the Dirac delta at $x$. If $\Omega = (0,1)$, then $\Omega_{\varepsilon} = (-\varepsilon, 1+\varepsilon)$ for $\varepsilon > 0$. We have $\inf\limits_{x\in\Omega} \nu(\Omega^c - x) = \frac {\pi^2}3$, however $\inf\limits_{x\in\Omega} \nu(\Omega_\varepsilon^c - x) = \frac {\pi^2}3 - 1$,  for every $\varepsilon > 0$, since we can take $x$ close to $0$ so that $1 \in \Omega_\varepsilon - x$. Thus $\lim\limits_{\varepsilon \to 0^+}\inf\limits_{x\in\Omega} \nu(\Omega^c_\varepsilon - x) < \inf\limits_{x\in\Omega} \nu(\Omega^c - x)$. One may easily check that in this setting, the solution to \eqref{exbarr} is $s(x) = \frac 3{\pi^2} \textbf{1}_{(0,1)} (x)$ hence the estimate \eqref{BJ} holds.
\end{exa}
\section{The extension operator}\label{sec:ext}
\subsection{Reflection in $C^{1,1}$ domains}
The construction of the extension consists of two main issues: choosing the method of the extension, and setting appropriate assumptions on the initial function. One may be tempted to extend the function simply by setting $u = 0$ in $\Omega$. The following example shows that a regular function, after being extended by $0$, may lose its good properties.
\begin{exa}
	Consider a one-dimensional L\'{e}vy measure $d\nu(x) = \frac 1{x^2} dx$ and a function $u = \frac 1x$ defined on $(-1,1)^c$. By the symmetry, we only need to perform the calculations on the positive half-line. We have
	\begin{align*}
	\int\limits_1^\infty\int\limits_1^\infty \left(\frac 1x - \frac 1y\right)^2 \frac 1{(x-y)^2} dx dy = \int\limits_1^\infty \int\limits_1^\infty \frac {(x-y)^2}{(xy)^2} \frac 1{(x-y)^2} dx dy = \int\limits_1^\infty \int\limits_1^\infty \frac 1 {(xy)^2} dx dy < \infty.
	\end{align*}
	If $\wtu$ is the function $u$ extended by $0$ to the whole of $\mR$, then
	\begin{align*}
	\int\limits_0^\infty \int\limits_0^\infty (\wtu(x) - \wtu(y))^2\frac 1{(x-y)^2} dx dy = &\int\limits_1^\infty \int\limits_1^\infty \left(\frac 1x - \frac 1y\right)^2 \frac 1 {(x-y)^2} dx dy\\ &+ 2 \int\limits_0^1\int\limits_1^\infty \frac 1{y^2} \frac 1{(x-y)^2} dy dx.
	\end{align*}
	Unfortunately, the second summand is infinite:
	\begin{align*}
	&\int\limits_1^\infty \frac 1{y^2} \int\limits_0^1 \frac 1 {(x-y)^2} dx dy = \int\limits_1^\infty \frac 1{y^2}\left(\frac 1{y-1} - \frac 1y\right)dy = \int\limits_1^\infty \frac 1 {y^3(y-1)} dy = \infty.
	\end{align*}
\end{exa}
This case shows, that the extension should be constructed in a more subtle way. Our method  -  the reflection can be used to obtain the extensions on $C^{1,1}$ domains, which are defined as follows
\begin{defi}\label{def1} An open, bounded, and connected $\Omega\subseteq \mR^n$ is a $C^{1,1}$ domain at scale $r>0$, if and only if it satisfies the interior and exterior ball conditions at some scale $r > 0$, i.e. for every $\wtx\in \partial \Omega$ there exist $x'\in\Omega^c$ and $x\in\Omega$, such that $B(x,r)\subseteq \Omega$, $B(x',r)\cap\Omega = \emptyset$, and $\overline{B(x,r)}\cap\overline{B(x',r)} = \{\wtx\}$.
\end{defi}
In \cite{Aikawa}, Aikawa et al. show that $C^{1,1}$ domains can be characterized as domains with the boundary that locally resembles the image of a $C^{1,1}$ function. To be precise, let $S_{\wtx}\partial\Omega$ be the plane tangent to $\Omega$ at $\wtx\in\partial\Omega$ and let $\vec{n_{\wtx}}$ be the normal vector at $\wtx$ (of whichever orientation).
\begin{thm}\label{def2}
	A domain $\Omega\subseteq\mR^n$ is $C^{1,1}$ at some scale $r>0$, if and only if there exist $\delta>0$, and $\lambda \geq 0$, such that for every $\wtx\in\partial\Omega$, $S_{\wtx}\partial\Omega$ exists, and there is a function $\Phi_{\wtx}:S_{\wtx}\partial\Omega \longrightarrow \mR^n$, given by the formula  $\Phi(x) = x + \phi(x)\vec{n_{\wtx}}$, such that $\phi: S_{\wtx}\partial\Omega \longrightarrow \mR$ is a $C^1$ function, and
	\begin{itemize}
		\item $|\nabla\phi(x) - \nabla\phi(y)| \leq \lambda |x-y|$ for every $x,y\in S_{\wtx}\partial\Omega$,
		\item $\Phi[B(\wtx,\delta)\cap S_{\wtx}\partial\Omega]\subseteq \partial\Omega$,
		\item $B(\wtx,\delta)\cap\partial\Omega \subseteq \Phi[B(\wtx,\delta)\cap S_{\wtx}\partial\Omega]$.
	\end{itemize}
\end{thm}
Note that in this setting $\phi(\wtx) = 0$ and  $\phi'(\wtx) = \textbf{0}$.

Later on, we will use $x$ ($x'$) to denote the center of an arbitrary interior (exterior) ball tangent to $\partial\Omega$ at $\wtx$. According to Definition \ref{def1}, consider a $C^{1,1}$ domain at scale $2r>0$. It is obvious that if $0<s<2r$, then $\Omega$ is also a $C^{1,1}$ domain at scale $s$. Note, that by taking exterior and interior balls of radius smaller than $2r$, we avoid the situation when one interior (or exterior) ball touches the boundary in more than one point. We also know that the center of the tangent ball lies on the line normal to $\Omega$ at $\wtx$. Thus, for every fixed $s\in(0,2r)$, we obtain a bijective correspondence between the center of the interior ball of radius $s$ and the point on the boundary that this ball is tangent to. We call that mapping $\psi_s: \partial\Omega \longrightarrow \Omega$. We also get a similar bijection for the center of the exterior ball: $\chi_s:\partial\Omega \longrightarrow \Omega^c$. The composition of these mappings is our desired reflection. 
\begin{defi}
	Let $\Omega$ be a $C^{1,1}$ domain, with constants $\lambda$, $\delta$ as in Theorem \ref{def2}, and $r$ according to Definition \ref{def1}. Let $V = \{x\in\mR^n: \dist(x,\partial\Omega) < \eps = r\wedge \frac 1{6\lambda}\wedge \frac {\delta}3 \}$. We define the reflection operator $T:V\longrightarrow V$ by the formulae $Tx = \chi_{d(x)} \circ \psi_{d(x)}^{-1} (x)$, for $x\in \Omega\cap V$, where $d(x) = \dist(x,\Omega^c)$, $Tx' = \psi_{d(x)}\circ \chi_{d(x)}^{-1} (x')$ for $x' \in Int(\Omega^c) \cap V$, and $T\wtx = \wtx$ for $\wtx\in\partial\Omega$.
\end{defi}
 From the construction we immediately get $T = T^{-1}$. The reasons for the choice of $\eps$ will be explained in the proof of Lemma \ref{dist}. This transformation, in general, does not preserve the distances between points, however we will prove that $|x-y| \approx |Tx - Ty|$ in $V$. In Figure \ref{pictr}, $x' = Tx$, $y' = Ty$.
\begin{figure}\label{pictr}
\begin{tikzpicture}[scale = 0.9]
\filldraw (0,0)  node[left] {$\wtx$} circle (1pt)
(8,5) node[above] {$\wty$} circle (1pt)
(-1.5,1.25) circle (1pt)
(1.5,-1.25) circle (1pt)
(7.75,6) circle (1pt)
(8.25,4) circle (1pt);

%\draw (0,0) coordinate (wtx) node[left] {$\wtx$} -- (8,5) coordinate (wty) node[right] {$\wty$};
\draw[red, thick] (0,0) .. controls (2.5,3) and (4,4) .. (8,5);
\draw (-1.5,1.25) node[left] {$x'$} circle (1.95cm);
\draw (1.5,-1.25) node[right] {$x$}circle (1.95cm);
\draw (7.75,6) node[above right] {$y'$} circle (1.03cm);
\draw (8.25,4) node[below right] {$y$} circle (1.03cm);
\draw (3,5.5) node[above left] {\Large{$\Omega^c$}};
\draw (4,3.5) node[above left] {\Large{\color{red}{$\partial\Omega$}}};
\draw (5, 1.5) node[above left] {\Large{$\Omega$}};
%\draw (0,0) -- (5,6);
%\draw (8,5) -- (0,3);
%\fill[green] (3.1579,3.7894) -- (3.4079,4.0894) arc (60:20:3.9mm) -- (3.1579,3.7894);
%\fill[blue] (8,5) -- (7.2,4.5) arc (240:219:9.4mm) -- (8,5);
\end{tikzpicture}
\caption{}
\end{figure}
\begin{lem}\label{dist} 
	There exists a constant $C\geq 1$, such that $|x-y| \leq C|Tx-Ty|$ holds for every $x,y\in V$. As a consequence $\frac 1C |x-y| \leq |Tx - Ty| \leq C|TTx - TTy| = C|x-y|$.
\end{lem}
\begin{proof}
	By writing $AB$ we mean the line segment with endpoints $A$ and $B$, $\Delta ABC$ is the triangle with vertices $A,B,C$. We will also use the same notations as before: $x\in\Omega$, $x' = Tx \in \Omega^c$ and $\wtx\in\partial\Omega$ is the midpoint of $x'x$. Let $U_x = \Phi[S_{\wtx}\partial\Omega\cap B(\wtx,\delta)].$
	It is enough to consider three cases: first - when both points are from $\Omega$, second - when one of them is in $\Omega$, and the other is in $\Omega^c$, and third - when one of the points is on the boundary.
	\paragraph{Case 1. $x,y \in\Omega$}
	\paragraph{Case 1.1. $\wty \in U_x$.}	
	\noindent We will assume that $\dist(y',\Omega) \leq \dist(x',\Omega)$.\\
	Let $z$ and $z'$ be the orthogonal projections of respectively $y$ and $y'$, on the unique line parallel to $x\wtx$ that goes through $\wty$. Furthermore, let $B'$ be the projection of $\wtx$ onto $zz''$ and $B$ - the projection of $\wty$ onto $xx'$, hence both $\Delta\wtx B \wty$ and $\Delta\wtx B'\wty$ are right triangles. See the illustration in Figure \ref{fig1}.
Note that the segment $y'y$ does not necessarily belong to the plane generated by the segments $x\wtx$ and $z\wty$. Our primary goal is to show that $|xy| \approx |x'y'|$. In order to do that we will first prove that $|xz| \approx |x'z'|$. By the Lagrange's mean value theorem $\frac {|B'\wty|}{|B'\wtx|} = |\phi'(\xi)|$ for some $\xi\in{\wtx B'}$. By the Lipschitz condition for the derivative, we get
\begin{equation}\label{Bysmall}
\frac {|B\wtx|}{|B\wty|} = \frac{|B'\wty|}{|B'\wtx|} = |\phi'(\xi)| \leq \lambda|B'\wtx|, 
\end{equation}
hence $|B\wtx| \leq \lambda|B\wty|^2$. 
	\begin{figure}
		\begin{tikzpicture}[scale = 0.8]
		\filldraw 
		(0,0) circle (1pt)
		(0,5) circle (1pt)
		(0,-5) circle (1pt)
		(8,-1) circle (1pt)
		(8,2) circle (1pt)
		(8,-4) circle (1pt)
		(9,2) circle (1pt)
		(7,-4) circle (1pt)
		(8,0) node[above left] {B'} circle(1pt)
		(0,-1) node[left] {B} circle(1pt);
		\draw (0,5)  node[left] {$x'$} -- (0,0) node[left] {$\wtx$} -- (0,-5) node[left] {$x$};
		\draw (0,0) -- (8,-1) node[above left] {$\wty$} -- (10,-1.25) ; 
		\draw (8,2)  node[above right] {$z'$}-- (8,-4) node[right] {$z$};
		\draw (9,2)  node[right] {$y'$} -- (7,-4) node[left] {$y$};
		\draw[dashed] (9,2) -- (8,2);
		\draw[dashed] (7,-4) -- (8,-4);
		\draw[dashed] (0,0) -- (8,0);
		\draw[dashed] (0,-1) -- (8,-1);
		%\fill[blue] (8,-1) -- (6,-1) arc (180:173:2cm) -- (8,-1);
		\end{tikzpicture}
		\label{fig1}\caption{Projection of $y'$ and $y$. $\angle (\wtx B \wty) = \angle (\wtx B' \wty) = \frac \pi 2$.}
	\end{figure} 
Assume, without loss of generality, that $|x'z'| \geq |xz|$. \\
	\begin{figure}
		\begin{tikzpicture}[scale=1.3]
		\filldraw (0,0) node[below] {$x'$} circle (1pt)
		(3,0) node[below] {$A$} circle (1pt)
		(4,0) node[below] {$\wtx$} circle (1pt)
		(5,0) node[below] {$B$} circle (1pt)
		%	  (7,0) node[below] {$C$} circle (1pt)
		(8,0) node[below] {$x$} circle (1pt)
		(3,4) node[above] {$z'$} circle (1pt)
		(5,4) node[above] {$\wty$} circle (1pt)
		(7,4) node[above] {$z$} circle (1pt);
		\draw (0,0) -- (8,0) -- (7,4) -- (3,4) -- (0,0);
		\draw (4,0) -- (5,4);
		%\draw[dashed] (3,0) -- (3,4);
		\draw[dashed] (5,0) -- (5,4);
		%\draw[dashed] (7,0) -- (7,4);
		\end{tikzpicture}
		\label{fig2}\caption{Projection of Figure \ref{fig1} on the plane. Here $Az'$ and $B\wty$ are the heights of the trapezoid.}
	\end{figure}
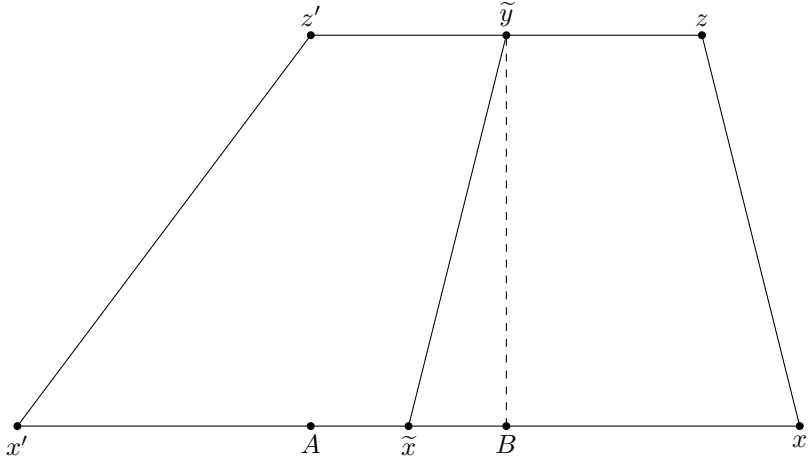
	The shape of the trapezoid may depend on positions of $\wtx$ and $\wty$, however the following arguments (especially, the formula for $|xz|$) are independent of this shape. Let $c = |x'x|-|z'z|$, $h = |B\wty|$, and $t = |x'A|$. Now $\frac{|xz|^2}{|x'z'|^2}$ can be represented as a function of $t$:
	\begin{equation}\label{repr}
	\frac {|xz|^2}{|x'z'|^2} = \frac {(t-c)^2 + h^2} {t^2 + h^2} = 1  - \frac{c(2t - c)}{t^2 + h^2}.
	\end{equation}
	Note that $|t - \frac c2| = |B\wtx|$. The assumption $|xz| \leq |x'z'|$ yields $2t - c \geq 0$. Hence, from \eqref{Bysmall}, we get $2t - c \leq 2\lambda h^2$. Therefore,
	\begin{equation}\label{rzuty}
	\frac {|xz|^2}{|x'z'|^2} \geq 1 - \frac {2c\lambda h^2}{t^2 + h^2} \geq 1 - 2c\lambda \geq \frac 12.
	\end{equation}
	Here we used that $c \leq |xx'|$, and $\eps \leq \frac 1 {6\lambda}$. Thus we have obtained that for some $D>0$, $|xz| \leq D|x'z'|$.\\
	Now we proceed to estimate $|xy|$.
	\begin{equation}\label{triangle}
	|xy| \leq |xz| + |zy| \leq D|x'z'| + |zy| \leq D|x'y'| + D|y'z'| + |zy| = D|x'y'| + (D+1)|z'y'|.
	\end{equation}
	We claim that $|z'y'| \leq a|x'y'|$ for some $a>0$ which does not depend on $x,y$. By the Lipschitz condition for $\phi'$, we get
	\begin{equation}\label{tan}
	\frac{|z'y'|}{|z'\wty|} = \tan (\angle y'\wty z') = |\phi'(\wtx) - \phi'(B')| \leq \lambda|B'\wtx| \leq \lambda|x'z'|.
	\end{equation}
	Therefore
	\begin{equation}\label{zyxz}
	|z'y'| \leq \lambda|x'z'||z'\wty| \leq \lambda |x'z'|d(y,\Omega) \leq \lambda\eps |x'z'|.
	\end{equation}
	Note that the we have made the assumption $\eps \leq \frac 1{2\lambda}$ in the definition of $V$. Hence
	\begin{equation}\label{lameps}
	|z'y'| \leq \frac 12 |x'z'|.
	\end{equation}
	By \eqref{lameps}, the triangle inequality, and \eqref{zyxz} we get the claim
	\begin{equation}\label{claim}
	|x'y'| \geq |x'z'| - |z'y'| \geq |x'z'| - \frac 12|x'z'| = \frac 12 |x'z'| \geq |z'y'|. 
	\end{equation}
	By applying \eqref{claim} to \eqref{triangle} we obtain
	\begin{equation*}
	|xy| \leq D|x'y'| + (D+1)|x'y'| = (2D+1)|x'y'|.
	\end{equation*}
	Thanks to $|xz| \approx |x'z'|$, the reverse estimate is obtained similarly, by interchanging $|xy|$ and $|x'y'|$ in \eqref{triangle}. Thus, Case 1.1. is proved.
	\paragraph{Case 1.2. $\wty \notin U_x$.} In that situation $|\wty\wtx| \geq \delta$. By the definition of $V$, we have $\eps < \frac{\delta}{3}$ and as a consequence $|xx'|,|yy'|\leq \frac{\delta}3$. Hence, $|x'y'| \geq |\wtx\wty| - |\wtx x'| - |\wty y'| \geq \delta - \frac \delta 3 - \frac \delta 3 = \frac \delta 3$. Analogously $|xy| \geq \frac \delta 3$. Since $|xy|$ and $|x'y'|$ are also bounded from above, we get $|xy| \approx |x'y'|.$ In the remaining cases we will not discuss the situation when $\wtx$ and $\wty$ are far from each other - they can be resolved in exactly the same way.
	\paragraph{Case 2. $x\in\Omega$, $y'\in\Omega^c$.}
	Once again we first project the situation on a plane with assumption that $|zz'| \leq |xx'|$ and $|x'z| \geq |xz'|$.
	\begin{figure}[H]
		\begin{tikzpicture}[scale=1.2]
		\filldraw (0,0) node[below] {$x'$} circle (1pt)
		(3,0) node[below] {$A$} circle (1pt)
		(4,0) node[below] {$\wtx$} circle (1pt)
		(5,0) node[below] {$B$} circle (1pt)
		%	  (7,0) node[below] {$C$} circle (1pt)
		(8,0) node[below] {$x$} circle (1pt)
		(3,4) node[above] {$z'$} circle (1pt)
		(5,4) node[above] {$\wty$} circle (1pt)
		(7,4) node[above] {$z$} circle (1pt);
		\draw (0,0) -- (8,0) -- (7,4) -- (3,4) -- (0,0);
		\draw (4,0) -- (5,4);
		%\draw[dashed] (3,0) -- (3,4);
		\draw[dashed] (5,0) -- (5,4);
		%\draw[dashed] (7,0) -- (7,4);
		\draw[red] (0,0) -- (7,4);
		\draw[red] (8,0) -- (3,4);
		\end{tikzpicture}
		\label{fig3}\caption{Illustration of the second case}
	\end{figure}
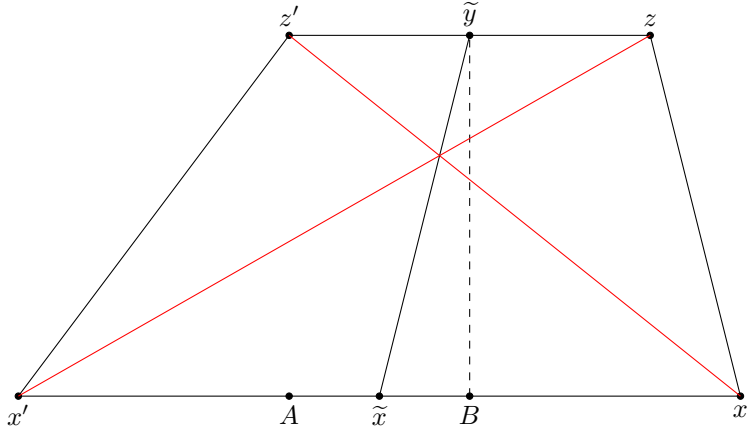
	We claim that $|x'z| \leq C|xz'|$ for some $C>0$ independent of $x,z$. Let $t,h,c$ be the same as before, and let $a = |z'z|$, $b = |x'x|$. Then, $|x'z|^2 = (t+a)^2 + h^2$, and $|xz'|^2 = (b-t)^2 + h^2$. Note that here we have the same condition on $t$ as in the previous case: $2t - c < 2\lambda h^2$. Therefore,
	\begin{align*}
	&|x'z|^2 - |xz'|^2 = 2(a+b)t + a^2 - b^2 = (a+b)(2t - (b-a)) = (a+b)(2t - c)\\ 
	&\leq (a+b)2\lambda h^2 \leq 8\eps\lambda  h^2 \leq 8\eps\lambda |xz'|^2.
	\end{align*}
	Thus we have obtained
	\begin{align}\label{diag}
	|x'z|^2 \leq |xz'|^2(1 + 8\eps\lambda).
	\end{align}
	The claim is proved. Note that in the last inequality of \eqref{tan}, we can change $|x'z'|$ to $|x'z|$. Therefore, to prove that $|x'y| \approx |xy'|$ we can use the same approach as in Case 1.1.  
	\paragraph{Case 3. $x\in\Omega$, $\wty \in \partial\Omega$.} In case 1.1., when proving that $|xz| = |x'z'|$ we could as well assume that $|yy'| = 0$. Therefore this situation can be handled in the same way. 
\end{proof}

\begin{cor}
	$T$ is a Lipschitz homeomorphism of $V$. In particular, $T$ maps Borel sets to Borel sets.
\end{cor}
\begin{lem}\label{mtk}
	Let $m$ be the Lebesgue measure on $\mR^n$. Then, there exists $C' \geq 1$, such that for every Borel $K\subseteq V$, we have $\frac 1{C'} m(K) \leq m(T[K]) \leq C' m(K)$.
\end{lem}
\noindent For the proof of this Lemma, see \cite{Naumann}, Theorem 3.1. Knowing that, we can deduce how the integrals behave under Lipschitz mappings.

\begin{cor}\label{jakob}
	The following "change of variable" formula holds for $W\subseteq V$:
	\begin{equation}\label{change}
	\pow g(Tx) dx \approx \potw g(x) dx,
	\end{equation}
	with the proportion being independent of $g$, and $W$.
\end{cor}
\begin{proof}
	By setting $Tx = y$, we have
	\begin{equation*}
	\pow g(Tx) dx = \potw g(y) d(m\circ T)(y).
	\end{equation*}
	From Lemma \ref{mtk}, we conclude that $m\circ T$ is absolutely continuous w.r.t Lebesgue measure thus, by Radon-Nikodym theorem, there exists $h\in L^1(V,dx)$, such that $d(m\circ T)(x) = h(x) dx$. Moreover, we have $0<h\leq C'$ a.e. on $V$. Hence,
	\begin{align*}
	\pow g(Tx) dx &= \potw g(y) d(m\circ T)(y) = \potw g(y) h(y) dy \leq C' \potw g(y) dy\\
	&= C' \pow g(Tx) d(m\circ T)(x) \leq (C')^2 \pow g(Tx) dx.
	\end{align*}
	
\end{proof}

\subsection{The extension operator}
For a function $u:D\to \mR$, we define the seminorm
$$\|u\|_{H_\nu(D)} = \sqrt{\pd\pd (u(x) - u(y))^2d\nu_x(y) dx}.$$
Let
\begin{equation}\label{hnud}
H_\nu(D) = \{u\in L^2(D): \|u\|_{H_\nu(D)} < \infty\}.
\end{equation}
$H_\nu(D)$ is a normed space with the norm $\|u\|_{H_\nu(D)} = \sqrt{\|u\|_{L^2(D)}^2 + \|u\|_{H_\nu(D)}^2}$. Note that for $D = \mR^n$, these definitions coincide with the ones from Section \ref{sec:func}.\\
From now on we will assume that $\Omega$ is a bounded $C^{1,1}$ domain. For a fixed $\Omega$ we define $W = V\cap \Omega$ with $V$ being the same as in Lemma \ref{dist}. Let $T$ be the reflection operator introduced in the previous section.
\begin{defi}
	Let $\phi\in C^{\infty}(\mR^n)$ satisfy $0\leq \phi \leq 1$, $\phi \equiv 1$ in $\Omega^c$, and $\phi \equiv 0$ in $\Omega\backslash W$. We define the extension operator $A:H_\nu(\Omega^c) \longrightarrow A[H_\nu(\Omega^c)]$ by the formula $A(g) = \wtg$, where
	$$\wtg(x) = \begin{cases}
	g(x) & \hbox{for } x\in\Omega^c, \\
	g(Tx)\phi(x) & \hbox{for } x\in W, \\
	0 & \hbox{for } x\in \Omega\backslash W.
	\end{cases}$$
\end{defi}
\begin{rem}
	In the work of Valdinoci et al. \cite{Sobolev}, the domain was assumed to be only $C^{0,1}$, i.e. Lipschitz. By assuming the ball condition we obtain a more transparent method of reflecting the function. In \cite{Zhou}, Zhou characterizes the domains in which the extension is possible in the context of fractional Sobolev spaces.
\end{rem}
\begin{thm}\label{ext}
	Let $\alpha \geq 1$ be the Lipschitz constant for the reflection operator $T$ ($\alpha$ depends only on $\Omega$). Assume that $\nu$ has an isotropic density $v(x) = V(|x|)$, for which there exists $C_\alpha$ such that for every $\beta \in [\alpha^{-1} \wedge \frac 13, \alpha]$, we have $V(\beta x) \leq C_\alpha V(x)$. Then $A$ is a continuous operator from $H_\nu(\Omega^c)$ to $\hrn$.
\end{thm}
	The extension problem for the fractional Laplacian is quite well-studied for $H_\nu(\mR^n)$ spaces \cite{Jonsson1978},\cite{Zhou}. Recently, Dyda and Kassmann \cite{BDMK} resolved the issue for spaces of type $\vkrn$ for $\nu$ corresponding to the fractional Laplacian. In our work, the extension belongs to the space $H_\nu(\mR^n)$. It is more restrictive due to the fact that we require the function to be "smooth" outside $\Omega$. However it allows us to use more general L\'evy measures.
\begin{proof}
	By assumptions, $d\nu(x) = v(x) dx$, where $v(x) = V(|x|)$ for some function $V: \mR_+ \longrightarrow [0,\infty)$. We have $g\in\homc$ i.e., $g\in L^2(\Omega^c)$ and $\poza\poza (g(x) - g(y))^2 d\nu_x(y) dx < \infty$.
	In order to show that $\wtg \in \hrn$, we need to show that $\wtg\in L^2(\mR^n)$ and $\porn\porn (\wtg(x) - \wtg(y))^2 d\nu_x(y) dx < \infty.$  For the first part, we have
	\begin{align}
	\porn \wtg(x)^2 dx &= \poza g(x)^2 dx + \int\limits_W \wtg(x)^2 dx = \poza g(x)^2 dx + \int\limits_W g(Tx)^2\phi(x)^2 dx\nonumber \\
	&\leq \poza g(x)^2 dx + \pow g(Tx)^2 dx \leq \poza g(x)^2 + C\potw g(x)^2 dx\label{L2} \\
	&\leq (C+1)\poza g(x)^2 dx, \nonumber
	\end{align}
	where in \eqref{L2} we have used Lemma \ref{jakob}. Thus, $\|\wtg\|^2_{L^2(\mR^n)} \leq C\|g\|_{L^2(\Omega^c)}^2 \leq C\|g\|_{\hrn}^2.$ \\
	We split the seminorm part into four integrals:
	\begin{align}
	\porn\porn (\wtg(x) - \wtg(y))^2d\nu_x(y)dx =& \poza\poza (g(x) - g(y))^2 d\nu_x(y) dx \label{triv}\\
	&+ \poza\poom (\wtg(x) - \wtg(y))^2 d\nu_x(y) dx \tag{A'}\label{A'} \\
	&+ \poom\poza (\wtg(x) - \wtg(y))^2 d\nu_x(y) dx \tag{A} \label{A} \\
	&+ \poom\poom (\wtg(x) - \wtg(y))^2 d\nu_x(y) dx \tag{B}\label{B}.
	\end{align}
	There is nothing to do in \eqref{triv}. Note that \eqref{A} = \eqref{A'} (cf. Lemma \ref{sym}). We will focus on \eqref{A}.
	\begin{align}
	\poom\poza (\wtg(x) - \wtg(y))^2 d\nu_x(y) dx =& \pow\poza (\wtg(x) - g(y))^2 d\nu_x(y) dx \tag{A.1}\label{A1} \\
	&+\int\limits_{\Omega\backslash W}\poza g(y)^2 d\nu_x(y) dx. \tag{A.2}\label{A2}
	\end{align}
	We have
	\begin{equation*}
	|g(Tx)\phi(x) - g(y)| \leq |g(Tx)\phi(x) - g(y) \phi(x)| + |g(y)\phi(x) - g(y)|,
	\end{equation*}
	hence \eqref{A1} is less or equal to
	\begin{align}
	&\pow\poza ( |g(Tx) - g(y)|\phi(x) + |g(y) - g(y) \phi(x)|)^2 d\nu_x(y) dx \nonumber \\
	&\leq 2\pow\poza g(y)^2(1 - \phi(x))^2 d\nu_x(y) dx \tag{A.1.1}\label{A11} \\
	&\quad + 2\pow\poza (g(Tx) - g(y))^2\phi(x)^2 d\nu_x(y) dx \tag{A.1.2}\label{A12}.
	\end{align}
	Note that smoothness of $\phi$ guarantees that there exists $C>0$, such that for every $y\in W$, we have $1 - \phi(y) \leq C \dist(y,\Omega^c)$. Therefore we can estimate \eqref{A11} as follows
	\begin{align*}
	\pow\poza g(y)^2 (1 - \phi(x))^2 v(x-y) dy dx &= \poza g(y)^2 \pow (1 - \phi(x))^2 v(x-y)dx dy\\
	&\leq C\poza g(y)^2 \pow \dist(x,\Omega^c)^2 v(x-y) dx dy. 
	\end{align*}
	Note that if $D = \eps \vee 1$, then we have $\dist(x,\Omega^c) \leq D(1\wedge |x-y|)$ for every $y\in\Omega^c$. Since $\nu$ is a L\'{e}vy measure, we get
	\begin{align*}
	\poza g(y)^2 \pow \dist(x,\Omega^c)^2 v(x-y)dx dy &\leq D^2\poza g(y)^2 \pow (1\wedge|x-y|^2) v(x-y) dx dy\\
	&\leq D^2 \left(\porn 1\wedge |x|^2 v(x)dx\right) \poza g(y)^2 dy.
	\end{align*} 
	From this we conclude that \eqref{A11} $\leq C \|g\|_{L^2(\Omega^c)}^2$.
	Substituting for $Tx$ in \eqref{A12} yields:
	\begin{align} \label{a12'}
	\pow\poza (g(Tx) - g(y))^2 \phi(x)^2 v(x-y) dy dx \approx \potw\poza (g(x) - g(y))^2\phi(Tx)^2 v(Tx - y) dy dx. 
	\end{align}
	We have $|x - y| \leq |x - Tx| + |Tx - y|$, and $|Tx - y| \geq d(Tx,\Omega^c) = \frac {|x-Tx|}2$, hence $|x-y| \leq 3|Tx - y|$. By the assumptions on $v$, we get $v(Tx-y) = V(|Tx - y|) \leq V(\frac 13|x-y|) \leq C_\alpha V(|x-y|)$. Therefore, the RHS of \eqref{a12'} is less or equal than
	\begin{align*}
	C_\alpha\potw\poza (g(x) - g(y))^2\phi(Tx)^2 v(x-y) dy dx \leq C_\alpha \|g\|_{\homc}^2.
	\end{align*}
	Estimation of \eqref{A2} is pretty straightforward. For this case, note that for every $y\in\Omega^c$, $\dist(y,\Omega\backslash W) > \eps$. Hence,
	\begin{align*}
	\poza g(y)^2 \int\limits_{\Omega\backslash W} v(x-y) dx dy = \poza g(x)^2 \nu((\Omega\backslash W)-x) dx \leq \nu(B(0,\eps)^c)\poza g(x)^2 dx.
	\end{align*}
	\eqref{B} can be split as follows
	\begin{align}
	\poom\poom (\wtg(x) - \wtg(y))^2 d\nu_x(y) dx &= \pow\pow (g(Tx)\phi(x) - g(Ty)\phi(y))^2 d\nu_x(y) dx \tag{B.1} \label{B1} \\
	&\quad +2\pow\int\limits_{\Omega \backslash W} g(Tx)^2\phi(x)^2 v(x-y) dy dx \tag{B.2} \label{B2}.
	\end{align}
	\eqref{B1} can be bounded from above by
	\begin{align}
	&2\pow\pow g(Tx)^2(\phi(x)-\phi(y))^2 d\nu_x(y) dx \tag{B.1.1} \label{B11} \\
	&+ 2 \pow\pow (g(Tx) - g(Ty))^2\phi(y) d\nu_x(y) dx. \tag{B.1.2} \label{B12}.
	\end{align}
	In \eqref{B11} we have
	\begin{align*}
	&\pow g(Tx)^2\pow (\phi(x) - \phi(y))^2 d\nu_x(y) dx \leq C\pow g(Tx)^2 \pow |x-y|^2 d\nu_x(y) dx \\
	&\leq C' \porn (1\wedge |y|^2) d\nu(y) \pow g(Tx)^2 dx \approx C'\porn (1\wedge |y|^2)d\nu(y) \potw g(x)^2 dx \\
	&\leq D \|g\|_{L^2(\Omega^c)}^2.
	\end{align*}
	We know that $\alpha|x-y| \geq |Tx - Ty| \geq \alpha^{-1}|x-y|$ holds for all $x,y \in W$. Using the densities properties we can estimate \eqref{B12}, which is less or equal to
	\begin{align*}
	\potw\potw (g(x) - g(y))^2 v(Tx - Ty) dy dx &\leq C_\alpha\potw\potw (g(x) - g(y))^2 v(x-y) dy dx \\
	&\leq C_\alpha \|g\|_{\homc}^2. 
	\end{align*}
	In order to estimate \eqref{B2}, note that for every $y\in\Omega\backslash W$, we have $\phi(x)^2 \apprle |x-y|^2$, thus \eqref{B2} is dominated by
	\begin{equation*}
	\pow g(Tx)^2 \int\limits_{\Omega\backslash W} |x-y|^2 d\nu_x(y) dx \apprle \porn(1\wedge |y|^2) d\nu(y)\pow g(Tx)^2 dx \approx \potw g(x)^2 dx \leq \|g\|_{L^2(\Omega^c)}^2.
	\end{equation*} 
	Summing up all the cases finishes the proof.
\end{proof}
\begin{rem}
	Extension, if it exists, is not determined uniquely. We can add any function from $\hkrn$ to it, which will not change the values outside $\Omega$.
\end{rem}
\begin{cor}\label{exicons}
Let $\nu$ and $\Omega$ satisfy the assumptions of Theorem \ref{ext}. If $g\in H_\nu(\Omega^c)$, then the Dirichlet problem \eqref{DP} has a unique weak solution.
\end{cor}
\section*{Acknowledgements} I would like to thank Krzysztof Bogdan for the great amount of discussions on the subject, especially for suggesting the representation in Lemma \ref{delty}. I also thank Bartłomiej Dyda, Tomasz Grzywny and Moritz Kassmann for their valuable remarks. I express my gratitude to the anonymous Referee, for the thorough review and important suggestions.
\bibliographystyle{plain}

\bibliography{bibliography}

\end{document}